\documentclass[10pt,reqno]{amsart}

 \usepackage{amsmath,amsfonts,amssymb,amscd,amsthm,amsbsy,bbm, epsf,calc,enumerate,color,graphicx,verbatim,cite,import}
\usepackage{color}
\usepackage{datetime,appendix}
\usepackage{chapterbib,epstopdf}

\usepackage{float}

\usepackage{ifpdf}
\ifpdf
    \usepackage[colorlinks,citecolor=black,linkcolor=black]{hyperref}
\fi

\theoremstyle{plain}
\newtheorem{definition}{Definition}[section]
\newtheorem{thm}{Theorem}[section]
\newtheorem{theorem}{Theorem}[section]

\newtheorem{lemma}[thm]{Lemma}

\newtheorem{proposition}[thm]{Proposition}
\theoremstyle{definition}

\theoremstyle{remark}                  %% For unnumbered Remarks, etc.

\DeclareMathOperator*{\essinf}{ess\,inf}

\definecolor{darkgreen}{rgb}{0,0.4,0}

\newcommand{\R}{\mathbb{R}}
\DeclareMathOperator*{\esssup}{ess\,sup}
\numberwithin{equation}{section}

\def \l {\left(}
\def\r {\right)}
\def\XXint#1#2#3{{\setbox0=\hbox{$#1{#2#3}{\int}$ }
\vcenter{\hbox{$#2#3$ }}\kern-.6\wd0}}

\title[Diffusion-Aggregation Equations]{On a Class of Diffusion-Aggregation Equations}
\author[Yuming Paul Zhang]{\bfseries Yuming Paul Zhang}

\address{
Department of Mathematics \\ % \hfill (Received 00 00 2010)\\
University of California   \\ %\hfill (Revised  00 00 2010)\\
Los Angeles\\
USA}
\email{yzhangpaul@math.ucla.edu}

\begin{document}

\vspace{18mm} \setcounter{page}{1} \thispagestyle{empty}

\begin{abstract}

{We investigate the diffusion-aggregation equations with degenerate diffusion $\Delta u^m$ and singular interaction kernel $\mathcal{K}_s = (-\Delta)^{-s}$ with $s\in(0,\frac{d}{2})$. We analyze the regime %($m>2-2s/d$, $d$ is the dimension)
where the diffusive forces are stronger than the aggregation forces. In such regime, we show existence, uniform boundedness and H\"{o}lder regularity of solutions in the case that either $s>\frac{1}{2}$ or $m<2$. Uniqueness is proved for kernels with $s>1$.}

\end{abstract}

\maketitle

\section{Introduction}

We consider the following degenerate diffusion equations with drifts:
\begin{equation}\label{main}
    u_t=\Delta u^m-\nabla\cdot(u \nabla \mathcal{K}_su)\text{ in }\mathbb{R}^d\times[0,\infty),
\end{equation}
with nonnegative initial data $u(x,0)=u_0 \in L^1(\mathbb{R}^d) \cap L^\infty(\mathbb{R}^d)$, where the degeneracy arises due to the range of $m$, $m>1$. 
The nonlocal drift is of the form
\begin{equation}
\label{eqn:K}
\mathcal{K}_su=c K_s*u   \quad\text{ where } s\in(0,\frac{d}{2}),\, K_s(z)=|z|^{-d+2s}.\\ %+V_1*u
\end{equation}
When $d\geq 3$, we can write $\mathcal{K}_su=(-\Delta)^{-s}u$ for some $c=c(d,s)>0$ 
which is a typical representation of the aggregating effect between density particles, with smaller $s$ representing stronger aggregation at near-distances and therefore more singular. For larger $s$, we consider stronger force at long-distances. In dimension two, the kernel of $(-\Delta)^{-s}$ is of a different form, for simplicity we restrict to $d\geq 3$.

\medskip

The model arises from the macroscopic  description of cell motility due to cell adhesion and chemotaxis phenomena, see \cite{carrillo2017ground,BCM}. The degenerate diffusion models the repulsion between cells to take over-crowding effects into consideration \cite{jacob60} and it is also widely seen in many physical applications, including fluids in porous medium. The homogeneous singular kernel models the attractive interactions between cells. The competition between non-local aggregation and diffusion is one of the core of subject of diffusion-aggregation equations. %The equation with the same range of $s$ is also studied in \cite{carrillo2017ground}.
%When $s=1$ our equation is the famous the Keller-Segel model of chemotaxis which is studied in a large number of literatures \cite{jacob33,32}. 

\medskip

To find the balance of diffusion and concentration effects, we use a scaling argument, also see \cite{critical10,critical}. Define \begin{equation}\label{def:ur}
    u_r(x,t):=r^d u(rx,r^{d(m-1)+2}t),\end{equation} 
then formally $(-\Delta)^{-s}u_r=r^{d-2s}(-\Delta)^{-s}u$. It is straightforward to check
\[\partial_t u_r=\Delta u_r^m-r^{2d-dm-2s}\nabla\cdot(u_r \nabla\cdot\mathcal{K}_su_r).\]
So $m=2-2s/d$ leads to a compensation effect between diffusion and aggregation. We call the range $m>2-{2s}/{d}$ {\it subcritical} where the diffusion dominates over the aggregation.  

%If we write $V(x,t)=-\nabla \mathcal{K}_s u(x,t)$, the motion of the solution is governed by $u_t=\Delta u^m+\nabla (uV)$. Notice by direct computation, \[\Delta K_s(z)=(-d+2s)(-2+2s)|z|^{-d-2+2s}\]which is locally integrable and negative when $s>1$. Since we only consider non-negative solutions, $\nabla\cdot V=-\Delta K_s*u\leq 0$ means that the vector field is contracting. Thus in order to take contracting drift into consideration, it makes sense for our discussion to include $\mathcal{K}_s u=-(-\Delta)^{-s}u$ when $s>1$. We will consider \begin{equation*}\mathcal{K}_s=\left\{\begin{aligned} &(-\Delta)^{-s}&\text{ when }&s\in (0,1]\\ &-(-\Delta)^{-s}&\text{ when }&s\in (0,\frac{d}{2}).\end{aligned}\right.\end{equation*}We comment that in the paper all statements and arguments work for $\mathcal{K}_s=\pm (-\Delta)^{-s}$. The situation that we have an opposite sign in front of $(-\Delta)^{-s}$ is actually easier.

\medskip

 When $s=1$, $\mathcal{K}_1$ represents the {\it Newtonian potential} and the equation \eqref{main} is the degenerate Patlak-Keller-Segel equation. In the range $m>2-2/d$, the well-posedness, boundedness and continuity regularity properties of solutions have been established, see \cite{bedrossian2011,chayes}. When $m=2-2/d$, it has been shown in \cite{calvez2017equilibria,critical14,critical} that there is a critical value of the mass and the behaviour of the solutions is determined by the initial mass. If the initial mass is large, solutions may blow up in finite time. If $m<2-2/d$, the aggregation dominates and the problem is called \textit{supercritical}, see \cite{bedrossian11,critical,bedrossian2011}. Again in this regime, we have finite time blow-up of solutions.

\medskip

In this paper, we consider the natural extension of the Newtonian potential with more near-range singularity, i.e. $\mathcal{K}_s=(-\Delta)^{-s}$ if $0<s<1$ and with more long-range singularity if $s>1$ (see \eqref{def:singular}). For this kernel, to the best of our knowledge,  only stationary solutions  have been analyzed before in \cite{carrillo2017ground}. It has been shown there that stationary solutions are radially symmetric decreasing with compact supports and enjoy certain regularity properties in most of the subcritical regime.   Our goal here is to initiate investigating the dynamic equation \eqref{main}, starting with its well-posedness and regularity properties. Many questions stay open as we discuss below.

% In \cite{arpotential}, equations with both attractive repulsive potentials has been studied while in our equation the potential is fully attractive. They assume that either the repulsive potential has a stronger singularity than the attractive one or in the fully attractive case, the potential has less singularity than the Newtonial potential.

\bigskip

\begin{flushleft}
{\bf Summary of our result.}
\end{flushleft}

%We summaries our results below.

\begin{theorem}[Existence and Boundedness Regularity]
Suppose the $d\geq 3$, $m,s$ are in the subcritical range and either $s>\frac{1}{2}$ or $m<2$. Let $u_0\in L^1(\R^d)\cap L^{\infty}(\R^d)$ be non-negative.  Then there exists a non-negative weak solution $u$ to \eqref{main} with mass preserved and $u$ is uniformly bounded for all $t\in [0,\infty)$. The bound only depends on $s,m,d, \|u_0\|_1$ and $\|u_0\|_\infty$.
\end{theorem}

We approach the problem by two approximations: regularization of the gradient of the kernel and elimination of the degeneracy, see \eqref{eqn:approx}. The key as well as the hard step is to show a prior boundedness estimates of solutions.
We will firstly prove uniform $L^p$-regularity properties (for some $p\to \infty$) of solutions to the approximate problems  in the subcritical regime. This can be seen as an variate result as compared to \cite{winkler1,winkler2,bedrossian2011} where Keller-Segel systems or equations are considered, see Theorems \ref{propm1}-\ref{thm:large s}. We are going to use Sobolev  inequalities, properties of fractional Laplacian and the equation to show some differential iterative inequalities which will eventually give us a uniform in time $L^\infty$ bound. The idea of the proof is to control the aggregation term by the degenerate diffusion. Very importantly in each estimate, we should not break the scaling \eqref{def:ur} and this turns out to be a useful hint for us, for example the choice of exponents in inequalities, see \eqref{def:pq}. And the condition $m>2-\frac{2s}{d}$ is essential in the proof.

In the subcritical range with $1/2<s\leq 1$, the uniform bounds are obtained separately when $2-2s/d<m<2$  and  $m\geq 2$, and only for the former range of $m$ when  $s\leq 1/2$. The proofs for the three cases are different. $s=\frac{1}{2}$ is critical, because  $|\nabla K_s|$ is only locally integrable when $s> \frac{1}{2}$. Boundedness of solutions in the case $\{m\geq 2, s\leq 1/2\}$ is unknown, though we believe it is true. While likely a technical challenge, extending the results seems to require some different ideas.  

When $s>1$, again the regimes $2-2s/d<m<2$  and  $m\geq 2$ are treated separately. However the proofs are even more different. In this regime the tool is limited, for example we can not use the fractional differentiation, instead we use Young's convolution inequality to treat the singular convolution integral. Technically we are required to use three arguments for different parts of the iterative steps, see Theorem \ref{thm:large s}.  

 Let us mention that the difficulty for $m>2$ arises as well in \cite{carrillo2017ground} where stationary states of \eqref{main} are studied. More precisely in Theorem 1 \cite{carrillo2017ground},  stationary solutions  are shown to be in $ W^{1,\infty}(\mathbb{R}^d)$ only when $2-2s/d < m\leq 2$.

\medskip

With aforementioned a priori bounds,  we obtain existence and bounds for the solution to the original problem \eqref{main} by compactness, see Theorem \ref{thmexistence}, \ref{thm:existence msmaller}. The hard part is to justify $u\nabla(-\Delta)^{-s}u=u\nabla K_s*u$ when $s\leq 1/2$, where $u$ is the weak solution of \eqref{main}, because in such cases $\nabla K_s*u$ is not well-defined for $u\in L^1\cap L^\infty$. To overcome this difficulty, the following estimate can be proved under the condition $m<2$,
\[\nabla u \in L^2_{loc}([0,\infty),L^2(\mathbb{R}^d) ).\]
Using this, we will show in Lemma \ref{lem:priorgrad} that 
\[\nabla(-\Delta)^{-s}u\in L^2_{loc}([0,\infty),L^2(\mathbb{R}^d)) \quad {\hbox{if }m<2}. \]

 %\textcolor{blue}{Actually when $s<1/2$, they have $\rho^{m-1}\in W^{1,\infty}(\mathbb{R}^d)$ for $2-\frac{2s}{d}<m<\frac{2-2s}{1-2s}$ in Theorem 17.}

%\begin{equation}\label{bootstrap}\rho^{m-1}=\frac{m-1}{m}((-\Delta)^{-s}\rho-C)_+\end{equation}
%

% We believe that for evolutionary equation \eqref{main}, $m=2$ is again critical in the regularity in some sense (which we don't know) and this in turn causes the difficulty of proving boundedness of solutions. 

\medskip

Next let us state the uniqueness result.

\begin{theorem}[Uniqueness]
Suppose $m,s$ are in the subcritical range and $s>1$, $u_0\in L^1(\R^d)\cap L^{\infty}(\R^d)$. Then the weak solution to \eqref{main} with initial data $u_0$ is unique. 

\end{theorem}

Uniqueness result is rather limited, it is only shown here for $s>1$ in the frame of $\dot{H}^{-1}(\R^d)$, following the approach of \cite{bertozzi2009,bedrossian2011}  where they consider the case when $s=1$. %One difference in the proof is to estimate $\left\|D^2 (-\Delta)^{-s} u\right\|_p$ for $s>1$ and all large $p$. 

\medskip

Now we look at the regularity of solutions to \eqref{main}.

\begin{theorem}[H\"{o}lder Regularity]\label{thm holder}
Suppose $m,s$ are in the subcritical range and $s\in(\frac{1}{2},\frac{d}{2})$. Let $u(\cdot,t)$ be a weak solution to \eqref{main} with non-negative initial data $u_0\in L^1(\R^d)\cap L^{\infty}(\R^d)$. Then the following holds
\begin{itemize}
\item[(a)] For any $\tau>0$, $u$ is H\"{o}lder continuous in $\mathbb{R}^d\times (\tau,\infty)$.

\medskip

\item[(b)] If $u_0$ is H\"{o}lder continuous in space, then $u$ is H\"{o}lder continuous in $\mathbb{R}^d\times [0,\infty)$.
\end{itemize}
\end{theorem}

When  $1/2<s<\frac{d}{2}$, $\nabla \mathcal{K}_su$ is a well-defined and bounded vector field in $\mathbb{R}^d$ for all $u\in L^1(\mathbb{R}^d)\cap L^\infty(\mathbb{R}^d)$ and then the interior H\"{o}lder estimate is a consequence of \cite{kimpaul} where the porous medium equation with locally integrable drift terms is considered. We also study the regularity of solutions on the boundary $t=0$ if the initial data is H\"{o}lder continuous which is given in Theorem \ref{tequal0}.

%We will take $V(x,t):=\nabla \mathcal{K}_su$ and the equation becomes \eqref{eqn:driftV}. Then with the help of prior estimates obtained, interior H\"{o}lder estimates (part (a)) come from \cite{kimpaul}. Part (b):  holds if given continuous initial data, see Theorem \ref{tequal0}.

The regularity result is left open in the regime $s\leq 1/2$ and the main difficulty comes from $\nabla(-\Delta)^{-s}u$. As mentioned before, we can have boundedness of $\nabla(-\Delta)^{-s}u$ in $L^2_{loc}([0,\infty),L^2(\mathbb{R}^d))$ when $m<2$. But this bound is not strong enough to obtain uniform H\"{o}lder estimates, according to \cite{kimpaul,chung2017holder}. We need some more careful analysis which could be employed  in future research.
%The difficulty of degenerate diffusion can be resolved by the method of Dibenedetto and Friedman in {\cite{dibenedetto1983continuity,dibenedetto2}}. 

%Then we also prove H\"{o}lder continuity of solutions up to time $0$ under the condition that the initial data is H\"{o}lder continuous. Then combining with the results stated above, in the subcritical case with $s>\frac{1}{2}$, solutions are globally uniformly H\"{o}lder continuous in both space and time. This will be used in the uniqueness part.

\medskip
Let us comment that our results and proofs adapt to more general kernels given by $ \mathcal{K}_su=K_s*u$ where $|{K}_s(x,y,t)|$, $|\nabla_x K_s(x,y,t)|$, $|D^2_x{K}_s(x,y,t)|$ share the same singularity as $|x-y|^{-d+2s}$, $|x-y|^{-d-1+2s}$, $|x-y|^{-d-2+2s}$ respectively near $x=y$. Some modifications are needed if we only assume $|K_s(x,y,t)|$, $|\nabla_x K_s(x,y,t)|$, $|D^2_x{K}_s(x,y,t)|$ to be bounded away from $x=y$. 

Lastly let us mention that a lot of open questions remain to be investigated even in the subcritical regime, existence result for $s<1/2$ (and $m>2$), uniqueness result for $s<1$. And there are even more questions in the supercritical regime. Some of these open questions closely related to us will be stated in the outline.

\bigskip

\begin{flushleft}
\textbf{Outline of the paper.}

\end{flushleft}

\medskip

We assume the space dimension $d\geq 3$ for the simplicity of computation, and also assume that  $m,s $ are in the subcritical range in the whole paper. Section 2 contains preliminary definitions and notations. Section 3 deals with  a priori estimates of solutions and the proof is given separately for $\{s\in (1/2,1],m<2\}, \{s>1/2,m\geq 2\}$, $\{s\leq 1/2, m<2\}$ and  $\{s> 1\}$. In section 4 we show the existence of solutions.  Existence and boundedness property of solutions for $\{s<1/2, m\geq 2\}$ stay open at the moment. In section 5, we give a uniqueness result for $s\geq 1$. The uniqueness problem is open for $s<1$. Lastly in section 6,  bounded solutions in the parabolic cylinder are shown to be H\"{o}lder continuous given that $s>1/2$. If given H\"{o}lder continuous initial data, solutions are H\"{o}lder continuous up to $t=0$.   Regularity property of solutions for $s\leq 1/2$ stays open. 

\leavevmode
\medskip

\textbf{Acknowledgements.}
The author would like to thank his advisor Inwon Kim for her guidance and stimulating discussions. The author would also like to thank Franca Hoffmann, Kyungkeun Kang and Monica Visan for helpful discussions and suggestions.

\medskip

\section{Preliminaries and Notations}\label{defassump}

\begin{flushleft}
$\circ$ Let us start with discussing the fractional potential operator $\mathcal{K}_s=(-\Delta)^{-s}$. 
\end{flushleft}
\medskip

We use the notation $-(-\Delta)^r$ with $r\in (0,1]$ for fractional Laplace operator which is defined on the Schwartz class of functions on $\mathbb{R}^d$ by Fourier multiplier with symbol $-|\xi|^{2r}$, see chapter V \cite{stein}. Alternatively, $-(-\Delta)^r$ can also be realized as the following singular integral in the sense of Cauchy principal value, see \cite{ten}. 
\[-(-\Delta)^r u(x) = \lim_{R\to 0^+}\frac{2^{2r}\Gamma(\frac{d+2r}{2})}{\pi^{d/2}|\Gamma(-r)|}\int_{\mathbb{R}^d\backslash B_R(x)}\frac{u(x+y)-u(x)}{|y|^{d+2r}}dy.
\]

We denote the constant before the above integral as $c_{d,r}$. The domain of the operator can be extended naturally to the Sobolev space $W^{2r,2}(\mathbb{R}^d)$. We will write 
\[|\nabla|^{2r}:=(-\Delta)^r.\]

The bilinear form associated to the space $W^{r,2}(\mathbb{R}^d)$ is define to be the following with reference to \cite{ten} and Section 3 \cite{caffarelli2012regularity}. For $v,w\in W^{r,2}(\mathbb{R}^d)$
$$
\mathcal{B}_r(v,w)=c_{d,r}\int \frac{(v(x)-v(y))(w(x)-w(y))}{|x-y|^{d+2r}}dxdy.
$$
Formally \[\mathcal{B}_r(v,w) = \langle (-\Delta)^{r} v, w\rangle_{L^2}:=\int_{\mathbb{R}^d}w(-\Delta)^{r} v\,dx.\]
%in the sense of Cauchy principal value.

Using Parseval's identity and definitions, we have for $0<r_1<r$
\[\langle (-\Delta)^{r} v, w\rangle_{L^2}=\langle|\nabla|^{r-r_1} v,|\nabla|^{r_1} w\rangle_{L^2}.\]

\begin{proposition}[Proposition 3.2 \cite{caffarelli2012regularity}]\label{propnonlocal} For every $v,w\in W^{1,2}(\mathbb{R}^d)$, we have\[\mathcal{B}_r(v,w)= C\int \frac{\nabla v(x)\cdot \nabla w(y) }{|x-y|^{d-2+2r}}dxdy.\]\end{proposition}

\medskip

The inverse operator is denoted by $-(-\Delta)^{-s}$ which can be realized as the convolution of a function with the Riesz potential 
\begin{equation}\label{def:singular}
(-\Delta)^{-s} u(x):=\int_{\mathbb{R}^d}K_s(x,y)u(y)dy;
\end{equation}
\[\text{ and }\, K_s(x,y):=\frac{2^{2s}\Gamma(\frac{d-2s}{2})}{{\pi^{d/2}}{\Gamma(s)}}|x-y|^{-d+2s}.\]
Here $s$ can be any number in $(0,\frac{d}{2})$ and $u$ is a function integrable enough for \eqref{def:singular} to make sense. We refer readers to \cite{ten,Caffarelli, Vazquez30} for more details.

%By the Fourier multiplier definitions, we have for any Schwartz function $u$ any $\alpha,\beta$,\begin{equation}\label{lem:hom}(-\Delta)^{\alpha}(-\Delta)^\beta u=(-\Delta)^{\alpha+\beta} u.\end{equation}

When $s>\frac{1}{2}$, $\nabla \mathcal{K}_s u$ is well defined for $u\in L^1(\mathbb{R}^d)\cap L^\infty(\mathbb{R}^d)$. When $s<\frac{1}{2}$, if we further assume that $u$ is $\gamma$-H\"{o}lder continuous with $\gamma\geq 1-2s$, then $\nabla \mathcal{K}_s u$ is defined via a Cauchy principal value
\[\nabla \mathcal{K}_s u(x):=\int_{\mathbb{R}^d}\nabla_x K_s(x,y)(u(y)-u(x))dy.\]

%Since $\mathcal{K}_s$ is positive self-adjoint operator, we write $\mathcal{H}_s=\mathcal{K}_s^{1/2}$ which has kernel $K_{s/2}$. For $u\in W^{1,2}$, $\mathcal{H}_s$ is a positive self-adjoint operator that commute with the gradient and we have\[\int_{\mathbb{R}^d}\nabla u\cdot \nabla\mathcal{K}_su dx=\int_{\mathbb{R}^d}|\nabla\mathcal{H}_su|^2dx.\]

\medskip

%as the completion of $C_0^\infty(\mathbb{R}^d)$ with respect to norm generated by the inner product\[(f,g)_{\dot{H}^1}:=\int_{\mathbb{R}^d}\nabla f(x)\cdot g(x)dx.\]We write $\dot{H}^{-1}(\mathbb{R}^d)$ as the dual space of $\dot{H}^1(\mathbb{R}^d)$.
%We will use $<,>$ to denote the dual pairing between them.

\begin{flushleft}
$\circ$ We now give the following notion of weak solutions to \eqref{main}. The notion is similar to the one in \cite{bertozzi2009,Caffarelli}.
\medskip
\end{flushleft}

\begin{definition}\label{def1.1}
Let $u_0(x)\in L^\infty(\mathbb{R}^d)\cap L^1(\mathbb{R}^d)$ be non-negative and $T\in (0,\infty]$. We say that a non-negative function $u:\mathbb{R}^d\times [0,T]\to[0,\infty)$ is a weak solution to \eqref{main} in time $[0,T]$ with initial data $u_0$ if
\begin{equation}\label{definitionsol}
\begin{aligned}
&u\in C([0,T],L^1(\mathbb{R}^d))\cap L^\infty(\mathbb{R}^d\times[0,T]),\quad u^m\in L^2(0,T, \dot{H}^1(\mathbb{R}^d)),\\
&\quad\text{ and } \quad u\nabla\mathcal{K}_su\in L^1(\mathbb{R}^d\times [0,T])
\end{aligned}
 \end{equation}
and for all test function $\phi\in C^\infty_c( \mathbb{R}^d\times [0,T))$, 
\begin{equation}\label{defofmain:test}\iint_{\mathbb{R}^d\times [0,T]} u\phi_tdxdt=\int_{\mathbb{R}^d} u_0(x)\phi(0,x)dx+\iint_{\mathbb{R}^d\times [0,T]}(\nabla u^m+u\nabla\mathcal{K}_su)\nabla\phi\, dxdt.\end{equation}
\end{definition}

%We say that $u$ solves \eqref{main} locally in an open set $S\subset\mathbb{R}^d$ in time $(0,T)$ if \eqref{definitionsol} holds and \eqref{defofmain:test} holds for all test function $\phi\in C^\infty_c( S\times (0,T))$.

%\begin{equation}\label{definitionsol}\begin{aligned}&u\in C([0,T],L^1(\mathbb{R}^d))\cap L^\infty(\mathbb{R}^d\times[0,T]),\quad u^m\in L^2(0,T, \dot{H}^1(\mathbb{R}^d)),\\&\quad\text{ and } \quadu\nabla\mathcal{K}_su\in L^1(\mathbb{R}^d\times [0,T]).\end{aligned} \end{equation}

\bigskip

\begin{flushleft}
$\circ$ Next we collect some known results which will be used later.
\end{flushleft}

%Recall the definition \textit{weak $L^p$ spaces}: a function $f$ is in the space $L^{p,\infty}$, if\[\sup_{\lambda>0} \lambda^{1/p}\left|\{x: |f(x)|\geq \lambda\}\right|<\infty.\]We write $\|f\|_{L^{p,\infty}}$ to be the supreme.

\begin{lemma}\label{lem:Young}[Young's convolution inequality]
For all $p,q,r\in [1,\infty]$ satisfying $1+1/q=1/p+1/r$, we have for all functions $f\in L^{p}(\mathbb{R}^d), g\in L^r(\mathbb{R}^d)$
\[\|f*g\|_{L^q}\leq\|f\|_{L^{p}}\|g\|_{L^r}.\]
\end{lemma}

%We use the notation\[|\nabla|^s u:= (-\Delta)^{s/2}u.\]

\begin{lemma}\label{gagliardo}
[Gagliardo-Nirenberg Interpolation Inequality]   
Let $\alpha,r,q,s$ be nonnegative constants satisfying
\begin{equation}\label{condition2}
 0\leq s\leq \alpha<1, \,1< r,\,p,\,q< +\infty
\end{equation}
\begin{equation}\label{condition1}
\text{ and }\quad\frac{1}{p}=\frac{s}{d}+\left(\frac{1}{r}-\frac{1}{d}\right)\alpha+\frac{1-\alpha}{q}. \end{equation}
For any function $u:\mathbb{R}^d\to\mathbb{R}$, suppose $u\in L^q(\mathbb{R}^d)$ and $\nabla u\in L^r(\mathbb{R}^d)$. 
Then there exists a constant $C$ depending only on $\alpha,r,q,s$ such that
\[\left\||\nabla|^s u\right\|_p\leq C\left\|\nabla u\right\|_r^\alpha\left\|u\right\|_q^{1-\alpha}.\]
Condition \eqref{condition2} can be replaced by \begin{equation}\label{condition3}0<s<\alpha<1,\,1<r, \, p<+\infty,\,1\leq q<+\infty.\end{equation}
If $s=0$, the inequality is classical and \eqref{condition2} can be replaced by \[s=0,\,0\leq \alpha\leq 1,\, 1\leq r,\,p\leq +\infty,\, 1\leq q<+\infty.\]

%For functions $u:B_R\to\mathbb{R}$, the interpolation inequality has the same hypotheses as above and reads\[\left\||\nabla|^s u\right\|_p\leq C_1\left\|\nabla u\right\|_r^\alpha\left\|u\right\|_q^{1-\alpha}+C_2\|u\|_1.\]where the constants $C_1, C_2$ are independent of $R$ for all $R$ large enough.
\end{lemma}
This lemma about Gagliardo-Nirenberg  Inequality is not given in the most general form, which is unnecessary for our purpose. 
We refer readers to \cite{nirenberg} for the classical Gagliardo-Nirenberg inequality. 
To the best of our knowledge, the validity of the inequality with fractional derivatives is proved in Corollary 1.5 \cite{BaoW}. But they did not cover the case when $q=1$ and \eqref{condition3} holds. We postpone the completion of the proof to the appendix. 

\medskip

The following lemma is useful which can be proved by using Calder\'on-Zygmund inequality. We refer readers to Theorem 4.3.3 \cite{grafakos} for the details.

\begin{lemma}\label{lem:absnabla}
There exists a constant $C>0$ such that for all $1<p<\infty$ and $u\in W^{1,p}$
\[\left\||\nabla|u\right\|_p\leq C\max\{p,\,(p-1)^{-1}\} \left\|\nabla u\right\|_p.\]
\end{lemma}

\medskip

Recall the \textit{homogeneous Sobolev space} $\dot{H}^s(\mathbb{R}^d)$:
\begin{definition}
Let $s\in\mathbb{R}$. The homogeneous Sobolev space is the space of tempered distributions $f$ over $\mathbb{R}^d$, the Fourier
transform of which belongs to $L^1_{loc}(\mathbb{R}^d)$ and satisfies
\[\|f\|_{\dot{H}^s}^2:=\int_{\mathbb{R}^d}|\xi|^{2s} |\hat{f}(\xi)|^2 d\xi<\infty.\]
\end{definition}

\begin{proposition}

\begin{flushleft}

If $|s|<\frac{d}{2}$, $\dot{H}^s$ can be considered as the dual space of $\dot{H}^{-s}$ through the bilinear functional: for any $f\in \dot{H}^s, g\in \dot{H}^{-s}$,
$(f,g)\to\int_{\mathbb{R}^d}f(x)g(x)dx.$

$\dot{H}^1$ is the subset of tempered distributions with
locally integrable Fourier transforms and such that $|\nabla f| \in L^2(\mathbb{R}^d)$.

\end{flushleft}
    
\end{proposition}
For details and more properties, we refer readers to the book \cite{bahouri2011fourier}.

%The following two lemmas are very useful which can be proved by using Calder\'on-Zygmund inequality. We refer readers to Theorem 4.3.3 \cite{grafakos} and Theorem 2.2 \cite{stein} for the details. \begin{lemma}\label{lem:absnabla}There exists a constant $C>0$ such that for all $1<p<\infty$ and $u\in W^{1,p}$\[\left\||\nabla|u\right\|_p\leq C\max\{p,\,(p-1)^{-1}\} \left\|\nabla u\right\|_p.\]\end{lemma}

\bigskip

\begin{flushleft}
\textbf{Notations.}
\end{flushleft}

\medskip

We write $\mathbb{N}$ as all natural numbers and $\mathbb{N}^+$ as all positive natural numbers.

For $p\geq 1, \gamma\in (0,1)$, for simplicity, we denote
$$\|\cdot\|_p:=\|\cdot\|_{L^p(\R^d)}\hbox{ and }\|\cdot\|_\gamma:= \|\cdot\|_{C^{\gamma}(\R^d)}.
$$
Given two points $(x,t),(y,s)\in\mathbb{R}^{d+1}$, we define the distance between them to be
\begin{equation}
    \label{distance in d plus one}
    |(x,t),(y,s)|:=|x-y|+|t-s|
\end{equation}
and we denote $|(x,t)|:=|(x,t),(0,0)|.$

For $j=d,d+1$, let $u:\mathbb{R}^j \to \R$ be a bounded measurable function and  $S$ be an open subset of $\mathbb{R}^j $. Denote
\[  osc_S(u):=\esssup_{x\in S}u-\essinf_{x\in S}u.\]

We write $B_R(x)$ as a ball in $\mathbb{R}^d$ centered at $x$ with radius $R$. We denote $B_R:=B_R(0)$.
The scaled parabolic cylinders are written as
\begin{equation*}Q(r,c):=\{x,|x|\leq r\}\times [-c{r^2},0] \hbox{ for } r,c>0. \end{equation*}
We denote the scaled parabolic cylinders near $t=0$ by
\begin{equation}\label{Qrc}Q^0(r,c):=\{x,|x|\leq r\}\times [0,c{r^2}] \hbox{ for } r,c>0. \end{equation}
The standard parabolic cylinders are denoted by $Q_r:=Q(r,1)$ and $Q^0_r:=Q^0(r,1)$.

\medskip

Throughout this paper, the constants $\{C\}$ represent  {\it universal constants}, by which we mean various constants that only depends on $m,d,s,\gamma$ and $L^1,L^\infty$ or $C^\gamma$ norms of the initial data $u_0$.  We may write $C(A)$ or $C_A$ to emphasize the dependence of $C$ on $A$.

We write 
$A\lesssim B$
if $A\leq CB$ for some universal constant $C$. When writing $A\lesssim_D B$, we mean $A\leq CB$ where $C$ depends on universal constants and $D$ (with particular emphasis on the dependence of $D$). By $A\sim B$, we mean both $A\lesssim B$ and $B\lesssim A$ are satisfied.

\medskip

Let $S$ be a measurable set in $\mathbb{R}^d$. The indicator function $\chi_S(x)$ equals $1$ if $x\in S$ and it equals $0$ otherwise.

\section{A Priori Estimates} \label{subsection4.1}

In this section several a priori estimates (mainly $L^\infty_tL^p_x$ and $L^\infty_tL^\infty_x$ bounds) are obtained. 

\medskip

We start with regularizing $\nabla K_s$ which is slightly different from the previous regularization.

Let us start with regularizing the equation \eqref{main} for $s\in (0,1]$.
Instead of modifying $\mathcal{K}_s$, we regularize $\nabla {K}_s$ by 
\begin{equation}
\label{def V}
V_{s,\epsilon }(x):=\zeta_\epsilon (x)\nabla_x {K}_s(x,0)
\end{equation}
where $\zeta_\epsilon $ (for some small $\epsilon >0$) is a smooth, radially symmetric, non-negative function that
\begin{align*}
   & \zeta_\epsilon=0 \text{ for }|x|\leq\epsilon \text{ and  }|x|\geq {2}/{\epsilon},\quad \zeta_\epsilon=1 \text{ for }|x|\in [2\epsilon,1/\epsilon],\\
   &\quad |\nabla\zeta_\epsilon |\lesssim 1/\epsilon \text{ for }|x|\leq 2\epsilon \quad \text{ and }\quad |\nabla\zeta_\epsilon |\lesssim \epsilon \text{ for }|x|\geq 1/\epsilon.
\end{align*} 
Then there is
\begin{itemize}
\item[1.] $V_{s,\epsilon }$ is a smooth vector field and $V_{s,\epsilon }(x)=c(-d+2s)|x|^{-d-2+2s}x$ for $|x|in [2\epsilon ,1/\epsilon]$;

\medskip

\item[2.] $|\nabla\cdot V_{s,\epsilon }(x)|\leq C|x|^{-d-2+2s}$ holds for all $x$ for some $C>0$ only depending on $d,s$.
\end{itemize}

%By simple calculation,  the kernel of $\mathcal{K}_{s,\mu}$ satisfies \[K_{s,\mu}(x,y):=\int_{\mathbb{R}^d}K(x,z)\zeta_\epsilon(z-y)dz. \]

%Next let $u_{0,\epsilon}$ be a non-negative, smooth and bounded approximation of the initial data $u_0$.
For small $\epsilon>0$, we consider $u_{\epsilon}$ which solves the following problem:
\begin{equation}\label{eqn:approx}
\left\{\begin{aligned}
&\frac{\partial}{\partial t}u_{\epsilon}=\epsilon \Delta u_{\epsilon}+\Delta u^m_{\epsilon}-\nabla\cdot (u_{\epsilon}V_{s,\epsilon}* u_{\epsilon})= 0 &\text{ in }\mathbb{R}^d\times [0,\infty), \\
& u_{\epsilon}(x,0)=u_{0}(x) &\text{ on } \mathbb{R}^d.
\end{aligned}
\right.
\end{equation}
%When $s>\frac{1}{2}$ the convolution $*$ is defined in the standard way by viewing $u$ as a function in $\mathbb{R}^d$ with $0$ extension outside $B_r$.
$V_{s,\epsilon}$ is smooth and compactly supported. The convolution integral $V_{s,\epsilon}*u$ is well-defined since $V_{s,\epsilon}$ bounded. 
Equation \eqref{eqn:approx} is uniformly parabolic. The existence and uniqueness of a solution $u_{\epsilon}$ is proved in Theorem 4.2 \cite{bertozzi2009} and the solution is smooth.

\medskip

In the following theorems, we are going to prove that $u_{\epsilon}$ are uniformly bounded independent of $\epsilon$. As mentioned before, we will treat the following cases separately: $\{s>\frac{1}{2},2-2s/d< m<2\}$, $\{s>\frac{1}{2}, m\geq 2\}$ and $\{s\leq \frac{1}{2}, 2-2s/d<m<2\}$. We use a refined iteration method and this approach can be found in Lemma 5.1 \cite{preventing}.

\begin{theorem}\label{propm1}
Suppose $d\geq 3$, $s\in(\frac{1}{2},1]$, $m\in(2-\frac{2s}{d},2)$ and $u_{0}\in L^1(\mathbb{R}^d)\cap L^\infty(\mathbb{R}^d)$ is non-negative. Let $u:=u_{\epsilon}$ be the solution to \eqref{eqn:approx}. Then there exists a constant $C$ such that for all $n\geq 1$ and $t\in \mathbb{R}^+$ there is
\[\left\|u\right\|_{L^n}(t)\leq C.\]
The constant $C$ only depends on $d,s,m$ and the $L^1,L^\infty$ norms of $u_{0}$.
\end{theorem}

\begin{proof}
Without loss of generality, let us suppose that the total mass of $u_{0}$ is $1$ and so is the total mass of $u(\cdot,t)$ by the equation. Since $u$ is smooth, for $n\geq 3-m$ we multiply $u^{n-1}$ on both sides of \eqref{eqn:approx} and find
\begin{align}
\partial_t \int_{\mathbb{R}^d} u^{n}dx&\leq -n \int_{\mathbb{R}^d} \nabla u^m\nabla u^{n-1}dx+n\int_{\mathbb{R}^d}(u V_{s,\epsilon}*u)\cdot\nabla u^{n-1}dx\nonumber\\
\label{ineq n2}&\leq-C_m\int_{\mathbb{R}^d}\left|\nabla u^{\frac{n+m-1}{2}}\right|^2dx+( n-1)\int_{\mathbb{R}^d}  V_{s,\epsilon}*u \nabla u^{n}dx.
\end{align}
By property 2. of $ V_{s,\epsilon}$, we obtain
\begin{align*}X&:=\int_{\mathbb{R}^d}  V_{s,\epsilon}*u \nabla u^{n}dx=\int_{\mathbb{R}^d} (-\nabla\cdot  V_{s,\epsilon})*u \; u^{n}dx\\
&\leq C \iint_{\mathbb{R}^{2d}}\frac{\left(u(x)-u(y)\right)\left(u^{n}(x)-u^{n}(y)\right)}{|x-y|^{d+2-2s}}dxdy  \quad  %(\text{ by Proposition~\ref{propnonlocal}}).
\end{align*}
Let $l=\frac{m+n-1}{2}<n$. Since $u$ is non-negative
\[\left(u(x)-u(y)\right)\left(u^{n}(x)-u^{n}(y)\right)\leq \left(u^l(x)-u^l(y)\right)\left(u^{n+1-l}(x)-u^{n+1-l}(y)\right).\]
Then
\begin{align*}
X&\leq C \iint_{\mathbb{R}^{2d}}\frac{\left(u^l(x)-u^l(y)\right)\left(u^{n+1-l}(x)-u^{n+1-l}(y)\right)}{|x-y|^{d+2-2s}}dxdy\\
&= C\int_{\mathbb{R}^d} \nabla (-\Delta)^{-s} u^l \nabla u^{n+1-l}dx \quad \text{  ( by Proposition \ref{propnonlocal})}\\
&\leq C \int_{\mathbb{R}^d}  \left|(-\Delta)^{1-s} u^l\right|  u^{n+1-l}dx
\\
&\leq C \left\||\nabla|^{2-2s}u^l\right\|_{2}\left\|u^{n+1-l}\right\|_{2}\quad \text{ ( by H\"{o}lder's inequality) }
\\
&=C \left\||\nabla|^{2-2s}u^l\right\|_{2}\left\|u^{l}\right\|_{2\frac{n+1-l}{l}}^\frac{n+1-l}{l}.\end{align*}
%Here $p_1,p_2$ satisfy $\frac{1}{p_1}+\frac{1}{p_2}=1$ which can be chosen to be $p_1=p_2=\frac{1}{2}$.
By Gagliardo-Nirenberg interpolation inequality
\[\left\||\nabla|^{2-2s}u^l\right\|_{2}\lesssim\left\|\nabla u^l\right\|_2^\alpha\left\|u^l\right\|^{1-\alpha}_{1},\]
\[\left\|u^l\right\|_{2\frac{n+1-l}{l}}\lesssim\left\|\nabla u^l\right\|_2^\beta\left\|u^l\right\|^{1-\beta}_{1}\]
with $\alpha(n),\beta(n)$ satisfying
\[\frac{1}{2}=\frac{2-2s}{d}+\l\frac{1}{2}-\frac{1}{d}\r \alpha+1-\alpha,\quad \frac{1}{2}\frac{l}{n+1-l}=\l\frac{1}{2}-\frac{1}{d}\r \beta+1-\beta.\]
It can be checked that $\alpha>2-2s$ if and only if $s>\frac{1}{2}$. The conditions of Lemma \ref{gagliardo} are satisfied.

Let $\theta(n):=\alpha+\frac{n+1-l}{l}\beta.$
Then the above two equalities give
\[\l\frac{1}{2}+\frac{1}{d}\r\theta(n)=\frac{2-2s}{d}+\frac{n+1-l}{l}=\frac{2-2s}{d}+\frac{4-2m}{n-1+m}+1.\]
Since $m<2$, $\{\theta(n)\}$ is decreasing as $n\rightarrow\infty$ and the limit equals $\l\frac{2-2s}{d}+1\r\Big/\l\frac{1}{2}+\frac{1}{d}\r$ which is less than $2$. Very importantly when $n=3-m$, $\theta({3-m})< 2$ is equivalent to
\[\l\frac{2-2s}{d}+3-m\r\Big/\l\frac{1}{2}+\frac{1}{d}\r<2 \iff m>2-\frac{2s}{d}.\]
So for all $n\geq 3-m$, $\theta=\theta(n)\in (\tau, 2-\tau)$ for some $\tau(m,s)>0$.
Then
\[X\leq \left\|\nabla u^l\right\|_2^{\theta}\left\|u^l\right\|_1^{1+\frac{n+1-l}{l}-\theta}.\]
By H\"{o}lder's inequality, for any small $\delta>0$
\begin{equation}\label{Xn}X\leq \frac{\delta}{n}\left\|\nabla u^l\right\|_2^2+C_{\delta}n^{c_n}\left\|u^l\right\|_1^{\theta'}\end{equation}
where
\[\theta'=\theta'(n)=2+\frac{2(2-m)}{l(2-\theta(n))}\leq 2+Cn^{-1}\quad \text{ and }\quad c_n=\frac{\theta(n)}{2-\theta(n)}.\]

Since $\theta(n)<2$ uniformly, $\{c_n\}$ are uniformly bounded in $n$ for all $n\geq 3-m$.
By Gagliardo-Nirenberg inequality
\[\left\| u^l\right\|_{\frac{n}{l}}\lesssim \left\|\nabla u^l\right\|^{{\gamma}}_2\left\|u^l\right\|^{1-{\gamma}}_{1}\text{ where } 
{\gamma}=\l\frac{n-l}{n}\r\Big/\l\frac{1}{2}+\frac{1}{d}\r.\]
By direct calculations, $\frac{{\gamma}n}{l}<2$. Then by Young's inequality
\begin{equation}\label{leftbound}
\int_{\mathbb{R}^d} u^n dx\lesssim \left\|\nabla u^l\right\|^{\frac{{\gamma} n}{l}}_2\left\|u^l\right\|^{\frac{(1-{\gamma})n}{l}}_{1}\lesssim  \left\|\nabla u^l\right\|_2^2+\left\| u^l \right\|_1^{\gamma'}
\end{equation}
where
\[\gamma'=\gamma'(n)=2-\frac{2(m-1)}{2l-{\gamma} n}\leq 2-C n^{-1}\text{ and (not hard to check) }\gamma'>0.\]

Finally by \eqref{ineq n2}\eqref{Xn}\eqref{leftbound}, we obtain for all $n\geq 3-m$
\[ \partial_t\int_{\mathbb{R}^d} u^n dx+c\int_{\mathbb{R}^d} u^n dx\leq Cn\l\int_{\mathbb{R}^d} u^ldx\r^{\gamma'}+Cn^{c_n+1}\l\int_{\mathbb{R}^d} u^ldx\r^{\theta'}
\]
where $c,C$ are independent of $n$.

\medskip

Now define a sequence $\{n_k, k\in \mathbb{N}^+\}\subset\mathbb{R}^+$ by
\begin{equation}\label{sequence}
    n_0=1,\quad n_{k+1}:=2  n_k+1-m \text{ for all }k\geq 0.
\end{equation}
Then $n_k=2^{k}(2-m)-1+m$. Since $m<2$, we have $n_k\to \infty$ as $k\to\infty$. 

\medskip

Notice for all $k\geq 1$, $n_k\geq n_1=3-m$.
Thus we can take $n=n_{k+1}$ in the above for all $k\geq 1$. Then $l=n_k$ and $n_k\sim 2^k$. If writing $A_k=\int_{\mathbb{R}^d} u^{n_k}dx$, we have proved for all $k\geq 1$
\[\frac{d}{dt}A_{k+1}+cA_{k+1}\leq C^k+C^k A_k^{2+O(2^{-k})}.\]
To conclude the proof we need the following lemma.
\end{proof}

\begin{lemma}[Lemma 3.1 \cite{kimpaul}]\label{iteration}
Suppose $n_0=1$ and $n_{k+1}:=2  n_k+a \text{ for all }k\geq 0$ with $a>-1$. Let $\{A_k(\cdot),k\in \mathbb{N}^+\}$ be a sequence of differentiable, positive functions on $[0,\infty)$ satisfying
\[\frac{d}{dt}A_k+C_0 A_k\leq C_1^{n_k}+{C_1}^k (A_{k-1})^{2+{C_1}n_k^{-1}},\]
for some constants $C_0,{C_1}$. 
Let $B_k(t):=A_k^{(n_k^{-1})}(t)$ and suppose  $\{B_0(t), B_k(0)\}$ are uniformly bounded with respect to $k\in\mathbb{N}, t>0$.
Then $\{B_k(t)\}$ are uniformly bounded for all $t> 0$ and $k\in \mathbb{N}^+$.
\end{lemma}
We refer readers to Lemma 3.1 \cite{kimpaul} for the proof.

\medskip

Now we consider the case when $m\geq 2$ and $s>\frac{1}{2}$.

\begin{theorem}\label{propm2}
Theorem \ref{propm1} holds in the regime: $m\geq2$, $s\in (\frac{1}{2},1]$.
\end{theorem}
\begin{proof}
%Let us only work with $\mathcal{K}*u=-(-\Delta)^{-s}u$.
Denote $u_j=\max\{u-j,0\}$, $\tilde{u}_j=\min\{u,j\}$ and so $u=u_j+\tilde{u}_j$. For some $n\geq 2$, let us multiply $u_1^{n-1}$ on both sides of \eqref{eqn:approx}. We have
\[\partial_t\int_{\mathbb{R}^d} u_1^n dx=
n\int_{\mathbb{R}^d} u_1^{n-1}u_tdx
\leq-mn\int_{\mathbb{R}^d} u^{m-1}\nabla u \nabla u_1^{n-1}dx+n\underbrace{\int_{\mathbb{R}^d}\l  V_{s,\epsilon}*u\r u\nabla u_1^{n-1}dx}_{X:=}.\]
Since
\[\int_{\mathbb{R}^d} u^{m-1}\nabla u \nabla u_1^{n-1}dx=\int_{\mathbb{R}^d} (u_1+1)^{m-1}\nabla u_1 \nabla u_1^{n-1}dx\]
\[\geq \frac{C_m}{n}\int_{\mathbb{R}^d}\left|\nabla u_1^{\frac{n+1}{2}}\right|^2+\left|\nabla u_1^{\frac{n}{2}}\right|^2dx\]
for some $ C_m>0 $ bounded from below for all $m\geq 2$, we obtain
\begin{equation}\label{ineq n}
\partial_t\int_{\mathbb{R}^d} u_1^n dx\leq
- C_m \int_{\mathbb{R}^d}\left|\nabla u_1^{\frac{n+1}{2}}\right|^2+\left|\nabla u_1^{\frac{n}{2}}\right|^2dx+nX.
\end{equation}

Let us now estimate $X$:
\begin{align}X&= \int_{\mathbb{R}^d}  V_{s,\epsilon}*u_1 u\nabla u_1^{n-1}dx+ \int_{\mathbb{R}^d} V_{s,\epsilon}*\tilde{u}_1\;u\nabla u_1^{n-1}dx\nonumber\\
&\lesssim \int_{\mathbb{R}^d}  V_{s,\epsilon}*u_1\nabla u_1^{n}dx+\int_{\mathbb{R}^d}  V_{s,\epsilon}*u_1\nabla u_1^{n-1}dx+\int_{\mathbb{R}^d}  V_{s,\epsilon}*\tilde{u}_1\;u\nabla u_1^{n-1}dx\nonumber\\
&=:Y_n+Y_{n-1}+X_{1}.\label{ineq:YYX}
\end{align}
We will first consider $X_{1}$. By the fact that \[s>\frac{1}{2},\; \tilde{u}_1\leq 1,\; \tilde{u}_1\in L^1 \text{ and } | V_{s,\epsilon}(x)|\lesssim |x|^{-d-1+2s},\] we have
\[ \left| V_{s,\epsilon}*\tilde{u}_1\right|(x)\lesssim\int_{\mathbb{R}^d}|x-y|^{-d-1+2s}\tilde{u}_1(y)dy\leq C\int_{\mathbb{R}^d} \tilde{u}_1 dy+\int_{|x-y|\leq 1} |x-y|^{-d-1+2s}dy\leq C.\]
Then for any small $\delta>0$
\begin{align}
\nonumber
X_{1}&= C\int_{\mathbb{R}^d} u\left|\nabla u_1^{n-1}\right|dx\lesssim \int_{\mathbb{R}^d} u\; u_1^{\frac{n}{2}-1}\left|\nabla u_1^\frac{n}{2}\right|dx\\ \nonumber
&\leq C_\delta n\int_{\mathbb{R}^d} \l u_1^n +u_1^{n-2}\r dx+\frac{\delta}{n}\|\nabla u_1^\frac{n}{2}\|_2^2\\
\label{ineq:X1'}
 &\leq C_\delta n\int_{\mathbb{R}^d}  u_1^n  dx+C_\delta n+\frac{\delta}{n}\|\nabla u_1^\frac{n}{2}\|_2^2.   
\end{align}
In the last inequality \eqref{ineq:X1'}, we applied
\[\int_{\mathbb{R}^d} u_1^{n-2}dx\leq \int_{u_1\geq 1} u_1^n dx+\int_{1\leq u\leq 2} 1 dx\leq \left\|u_1^\frac{n}{2}\right\|_2^2+1.\]

Next by Gagliardo-Nirenberg and Young's inequalities
\begin{equation}\label{u1n2}
   C_\delta\left\|u_1^\frac{n}{2}\right\|^2_{2}\leq C_\delta C_\alpha\left\|\nabla u_1^\frac{n}{2}\right\|_2^{2\alpha}\left\|u_1^\frac{n}{2}\right\|_1^{2(1-\alpha)} \leq C_{\delta}'n^{d}\left\|u_1^\frac{n}{2}\right\|^2_1+\frac{\delta}{n^2} \left\|\nabla u_1^\frac{n}{2}\right\|^2_2
\end{equation}
where we picked
\[\alpha={\frac{1}{2}}\Big/\l{\frac{1}{2}+\frac{1}{d}}\r.\]%,\quad c_0=\frac{2\alpha}{1-\alpha}=
So by \eqref{ineq:X1'}, for some universal small $\delta>0$ 
\begin{equation}\label{ineq:X_1}X_1\leq C_\delta n+C_{\delta}' n^{d+1}\left\|u_1^\frac{n}{2}\right\|^2_1+\frac{\delta}{n} \left\|\nabla u_1^\frac{n}{2}\right\|^2_2.\end{equation}

For $Y_l$ with $l=n-1,n$, as proved before (in Theorem \ref{propm1})
\[Y_l\lesssim \iint_{\mathbb{R}^{2d}}\frac{\left(u_1^{\frac{l+1}{2}}(x)-u_1^{\frac{l+1}{2}}(y)\right)^2}{|x-y|^{d+2-2s}}dxdy\lesssim \int_{\mathbb{R}^d} \nabla(-\Delta)^{-s}u_1^{\frac{l+1}{2}}\nabla u_1^\frac{l+1}{2}dx.\]
By Fourier transformation and H\"{o}lder's inequality, 
\begin{align}\nonumber
Y_l &\lesssim \int_{\mathbb{R}^d} \left|\xi\right|^{2-2s}\left|\widehat{u_1^{\frac{l+1}{2}}}\right|^2 d\xi\lesssim \l\int_{\mathbb{R}^d}|\xi|^2\left|\widehat{u_1^\frac{l+1}{2}}\right|^2d\xi\r^{1-s}\l\int_{\mathbb{R}^d} \left|\widehat{u_1^\frac{l+1}{2}}\right|^2d\xi\r^{s}\\\nonumber
&\lesssim \left(\int_{\mathbb{R}^d}\left|\nabla u_1^{\frac{l+1}{2}}\right|^2dx\right)^{1-s}\left(\int_{\mathbb{R}^d} u_1^{l+1}dx\right)^{s}\\\nonumber
&\leq C_\delta n^{\frac{1-s}{s}}\left\|u_1^\frac{l+1}{2}\right\|_2^2+\frac{\delta}{n}\left\|\nabla u_1^\frac{l+1}{2}\right\|_2^2\\\label{ul1}
&\leq C_\delta n\left\|u_1^\frac{l+1}{2}\right\|_2^2+\frac{\delta}{n}\left\|\nabla u_1^\frac{l+1}{2}\right\|_2^2.\end{align}
We used $s> \frac{1}{2}$ in the last inequality.

When $l=n-1$, by \eqref{ul1}\eqref{u1n2}, we have for some $C$ only depending on $\delta$
\[Y_{n-1}\leq C n^{d+1}\left\|u_1^\frac{n}{2}\right\|_1^2+\frac{\delta}{n}\left\|\nabla u_1^\frac{n}{2}\right\|_2^2.\]

When $l=n$, as done previously
\[Y_{n}\leq C n^{d+1}\left\|u_1^\frac{n+1}{2}\right\|_1^2+\frac{\delta}{n}\left\|\nabla u_1^\frac{n+1}{2}\right\|_2^2.\]
By Gagliardo-Nirenberg,
\[\left\|u_1^\frac{n+1}{2}\right\|^2_1=\left\|u_1^\frac{n}{2}\right\|^{2\frac{n+1}{n}}_{\frac{n+1}{n}}\leq C\left\|\nabla u_1^\frac{n}{2}\right\|_2^{2\beta_1}\left\|u_1^\frac{n}{2}\right\|_1^{2\beta_2} 
\]
where \[\beta_1=\beta_1(n)=\frac{1}{n}\Big/\l{\frac{1}{2}+\frac{1}{d}}\r,\quad \beta_2=\beta_2(n)=\frac{n+1}{n} -\beta_1. \]
By Young's inequalities
\begin{align*}
    C\left\|\nabla u_1^\frac{n}{2}\right\|_2^{2\beta_1}\left\|u_1^\frac{n}{2}\right\|_1^{2\beta_2} 
    &\leq C( \frac{\epsilon^p\|\nabla u_1^\frac{n}{2}\|^{2\beta_1 p}_2}{p}+\frac{\|u_1^\frac{n}{2}\|_1^{2\beta_2q}}{\epsilon^qq})\\
    &\leq C_{\delta}n^{c_n}\left\|u_1^\frac{n}{2}\right\|^{\gamma_n}_1+\frac{\delta}{n^{d+2}} \left\|\nabla u_1^\frac{n}{2}\right\|^2_2
\end{align*}
where we pick \[p=\frac{1}{\beta_1}\sim n,\quad q=\frac{p}{p-1},\quad C\epsilon^p/p=\frac{\delta}{n^{d+2}}.\] Thus $\epsilon^p\sim \frac{\delta}{n^{d+1}}$ and since $\frac{-q}{p}=-\frac{1}{p-1}=-\frac{\beta_1}{1-\beta_1}$,
we have $\epsilon^{-q}= C_\delta n^{\frac{\beta_1(d+1)}{1-\beta_1}}$. We have 
\[\gamma_n=\frac{2\beta_2}{1-\beta_1}=2+\frac{1}{n}\frac{2}{(1-\beta_1)}\;\text{ and }\; c_n=\frac{\beta_1(d+1)}{1-\beta_1}.\]
It is not hard to check that for all $n\geq 2$, $\beta_1(n)\leq \beta_1(2)<1$. And so $c_n$ is uniformly bounded for all $n\geq 2$. So we proved that for any small $\delta>0$
\[Y_n\leq C_\delta n^{c_n}\left\|u_1^\frac{n}{2}\right\|^{\gamma_n}_1+\frac{\delta}{n} \left\|\nabla u_1^\frac{n}{2}\right\|^2_2+\frac{\delta}{n} \left\|\nabla u_1^\frac{n+1}{2}\right\|^2_2.\]

Combining with \eqref{ineq:YYX} and \eqref{ineq:X_1}, for some $c$ $(=c_n+d+1)>0$ uniformly bounded for all $n\geq 2$
\[
nX\leq C_\delta n^2+C_\delta n^c\l\left\|u_1^\frac{n}{2}\right\|^2_1+\left\|u_1^\frac{n}{2}\right\|^{{\gamma_n}}_1\r+\delta\left\|\nabla u_1^\frac{n}{2}\right\|_2^2+\delta\left\|\nabla u_1^\frac{n+1}{2}\right\|_2^2.\]
Therefore by \eqref{ineq n}
\begin{equation}\label{eqn:diffu1}\frac{d}{dt}\left\|u_1^n\right\|_1+\l\left\|\nabla u_1^\frac{n}{2}\right\|_2^2+\left\|\nabla u_1^\frac{n+1}{2}\right\|_2^2\r\lesssim n^2+n^c\l \left\|u_1^\frac{n}{2}\right\|^2_1+\left\|u_1^\frac{n}{2}\right\|^{{\gamma_n}}_1\r. \end{equation}

Again by Galiardo-Nirenberg inequality and Young's inequality
\[\left\|u_1^\frac{n}{2}\right\|_{2}\lesssim \left\|\nabla u_1^\frac{n}{2}\right\|^\theta_2\left\|u_1^\frac{n}{2}\right\|_1^{1-\theta}+\left\|u_1^\frac{n}{2}\right\|_1\lesssim
\left\|\nabla u_1^\frac{n}{2}\right\|_2+\left\|u_1^\frac{n}{2}\right\|_1 \]
where $\theta=\frac{1}{2}/\l\frac{1}{2}+\frac{1}{d}\r$. So for some universal $C,c>0$
\begin{equation}\label{noninter3}
\left\|\nabla u_1^\frac{n}{2}\right\|^2_2\geq C\left\|u_1^\frac{n}{2}\right\|_2^2-c\left\|u_1^\frac{n}{2}\right\|_{1}^2.
\end{equation} 
By \eqref{eqn:diffu1}, \eqref{noninter3}, we have
\[\frac{d}{dt}\left\|u_1^n\right\|_1+\left\|u_1^n\right\|_1 \lesssim n^2+n^c\left\|u_1^\frac{n}{2}\right\|^{{\gamma_n}}_1. \]
Recall here $\gamma_n\leq 2+\frac{C}{n}$.

Now we let $n=2^k$ for $k=1,2..$ and $A_k=\int_{\mathbb{R}^d} u_1^{n_k}dx$. By Lemma \ref{iteration}, $u_1(x,t)$ is uniformly bounded for all $t\geq 0$ and so is $u(x,t)$.
\end{proof}

%Let us prove for small $m$.We will always take $V(x)=-\frac{1}{|x|^{d-2s}}$. We separate theorem into two cases of $s>\frac{1}{2}$ and $s<\frac{1}{2}$.

Now we turn to the case when $s\in (0,\frac{1}{2}]$ and $m\in (2-2s/d,2)$.

\begin{theorem}\label{thm:s1/2}
Theorem \ref{propm1} holds in the regime:  $s\in (0,\frac{1}{2}]$, $m\in (2-2s/d,2)$.
\end{theorem}

\begin{proof}

For all $n\geq 3-m$, let $l=\frac{m+n-1}{2}$. Multiplying $u^{n-1}$ on both sides of \eqref{eqn:approx}, we obtain
\begin{equation}\label{ineq n3}
\partial_t \int_{\mathbb{R}^d} u^{n}dx\leq 
-C_m\int_{\mathbb{R}^d}|\nabla u^l|^2dx+C n\int_{\mathbb{R}^d}  V_{s,\epsilon}*u \nabla u^{n}dx.
\end{equation}
By Proposition \ref{propnonlocal}, properties of fractional derivatives and Young's inequality, for any $p>1,q>1$ satisfying $\frac{1}{p}+\frac{1}{q}=1$
\begin{align*}
    &\quad\int_{\mathbb{R}^d}  V_{s,\epsilon}*u \nabla u^{n}dx=\int_{\mathbb{R}^d} (-\nabla\cdot  V_{s,\epsilon})*u \; u^{n}dx\\
    &\lesssim \int_{\mathbb{R}^d} \nabla (-\Delta)^{-s}u \nabla u^{n}dx\lesssim \int_{\mathbb{R}^d} \nabla (-\Delta)^{-s}u^{l} \nabla u^{n+1-l}dx\\
&\lesssim \int_{\mathbb{R}^d} \left( (-\Delta)^{1-s}u^{l} \right) u^{n+1-l}dx\lesssim \int_{\mathbb{R}^d} \left||\nabla|^{1-2s}u^{l} \right| \left||\nabla|u^{n+1-l} \right| dx\\
&\lesssim n^{p/q}\left\||\nabla|^{1-2s} u^l \right\|_{p}^p+\frac{1}{n}\left\| |\nabla|u^{n+1-l}\right\|_q^q\\
&:=n^{p/q}X^p_1+\frac{1}{n}X_2.
\end{align*}

To choose $p$ and $q$, we rescale the solution by setting $u_r=r^d u(rx,t)$. Then 
\[\left\||\nabla|^{1-2s} u_r^l \right\|_{p}^p=\left(r^{1-2s}r^{dl}\right)^p\left\||\nabla|^{1-2s} u^l \right\|_{p}^p;\]
\[\left\| |\nabla|u_r^{n+1-l}\right\|_q^q=\left(r r^{d(n+1-l)}\right)^q \left\| |\nabla|u^{n+1-l}\right\|_q^q.\]
To match the scaling, we want $(1-2s+dl)p=(1+d(n+1-l))q$ in the case when $m=2-\frac{2s}{d}$. And using $\frac{1}{p}+\frac{1}{q}=1$, we find
\begin{equation}\label{def:pq}p(n)=\frac{2+d(m+n-1)}{1+d(m+l-2)},\,q(n)=\frac{2+d(m+n-1)}{1+d(n+1-l)}.\end{equation}
These are the values we pick for $p,q$. When $n=3-m,l=1$, we obtain
\begin{equation}\label{def:pqsimple}p(3-m)=\frac{2+2d}{1+d(m-1)},\,q(3-m)=\frac{2+2d}{1+d(3-m)}.\end{equation}
While as $n\to\infty$ ($l=\frac{n+m-1}{2}$), $p(n)$ is monotonically decreasing, $q(n)$ is monotonically increasing and
\[p(n)\to 2,\, q(n)\to 2.\]
Also it is not hard to see that
\[p(n)>1,\,q(n)>1,\,1\leq\frac{p(n)}{q(n)}\leq  \frac{1+d(3-m)}{1+d(m-1)}=:c_1(d,m),\]
\begin{equation}
    \label{con:pn}
    p(n)-2=\frac{2d(2-m)}{1+d(m+l-2)}\sim \frac{1}{n},
\end{equation}
\begin{equation}
    \label{con:pqn}
    2-q(n)=\frac{2d(2-m)}{1+d(n+1-l)}\sim \frac{1}{n}.
\end{equation}

By Lemma \ref{lem:absnabla} and Young's inequality, for any $\delta\in(0,1)$
\begin{align*}
    X_2&\leq C\left\| \nabla u^{n+1-l}\right\|_q^q= C' \int_{\mathbb{R}^d} u^{(n+1-2l)q}\left|\nabla u^l\right|^q dx\\
&\leq C_\delta n  \left\|u^{(n+1-2l)q}\right\|_{\frac{2}{2-q}}^{\frac{2}{2-q}}+{\delta}\left\|\left|\nabla u^l\right|^q\right\|_{\frac{2}{q}}^{\frac{2}{q}}\\
&= C_\delta  n
\left\|u^{l}\right\|^2_{2}+{\delta}\left\|\nabla u^l\right\|^2_2.
\end{align*}
In the last inequality, we used \eqref{con:pqn}.
Now by Gagliardo-Nirenberg interpolation inequality
\[C_\delta
\left\|u^{l}\right\|^2_{2}\leq C_\delta \left\|\nabla u^{l}\right\|^{2\beta}_{2}\left\|u^{l}\right\|^{2(1-\beta)}_{1}\leq \frac{\delta}{n} \left\|\nabla u^{l}\right\|^{2}_{2}+C_\delta n^{c_0}\left\|u^{l}\right\|^{2}_{1}\]
where
\[\beta=\frac{1}{2}\big/(\frac{1}{2}+\frac{1}{d}),\, c_0=\frac{\beta}{1-\beta}=\frac{d}{2}.\]
So
\begin{equation}
    \label{ineq:X2}
    X_2\leq C_\delta  n^{c_0+1}
\left\|u^{l}\right\|^2_{1}+2{\delta}\left\|\nabla u^l\right\|^2_2.
\end{equation}

\medskip

Next for $X_1$, again by Gagliardo-Nirenberg interpolation inequality
\[X_1=\left\||\nabla|^{1-2s} u^l\right\|_{p}\lesssim \left\|\nabla u^l\right\|^\alpha_2\left\|u^l\right\|^{1-\alpha}_{1}
\]
where we need to put
\[
\alpha=\alpha(n)=\l{\frac{1-2s}{d}-\frac{1}{p}+1}\r\Big/{\l\frac{1}{d}+\frac{1}{2}\r}.\]
It is not hard to check that $s< \alpha(n)<1$ uniformly for all $n\geq 3-m$. Moreover, we claim that $\sup_{n\geq 3-m}\alpha(n) p(n)<2$.
We only need to check when $n=(3-m)$ by monotonicity of $\alpha(n)p(n)$ in $n$.  By \eqref{def:pqsimple} and direct calculations, 
\[\alpha \frac{2+2d}{1+d(m-1)}<2 \iff m>2-\frac{2s}{d}.\] 
With this, we obtain 
\begin{equation}
    \label{ineq:X1}
    X_1^p\lesssim \frac{\delta}{n^{c_1+1}}\|\nabla u^l\|_2^2+C_\delta n^{c_2}\|u^l\|^{2+\gamma}_1
\end{equation}
where $c_2$ is a constant that
\[c_2\geq (c_1+1)\frac{\alpha(1-\alpha)p^2}{2-\alpha p}\]
and by \eqref{con:pn}
\begin{equation}
    \label{def:gamma}
    \gamma=\gamma(n)=\frac{2(p-2)}{2-\alpha p}\sim \frac{1}{n}.
\end{equation}

Putting together \eqref{ineq:X2} and \eqref{ineq:X1} shows
%\[X\lesssim\left\|\nabla u^l\right\|^{2\alpha}_2\left\|u^l\right\|^{2(1-\alpha)}_{1}\|u\|_1^{\frac{1}{q_1}}\leq \frac{\delta}{n^2}\left\|\nabla u^l\right\|_2^2+C_\delta n^{c_0} \left\|u^l\right\|_1^2.\]
%Here we used that $\|u\|_1=\|u_{0,\epsilon}\|_1$ which is a constant and Young's inequality. The constant $c_0$ satisfies
%\[c_0(s,d,m)=\frac{2\alpha}{1-\alpha}\]
%which only depends on $s,d,m$.Putting together what we have,
\begin{align*}
    &\quad n\int  V_{s,\epsilon}*u \nabla u^{n}dx\leq n^{c_1+1}X_1^p+X_2\\
&\leq C_\delta n^{c_2+c_1+1}\|u^l\|^{2+\gamma}_1+C_\delta  n^{c_0+1}
\left\|u^{l}\right\|^2_{1}+3{\delta}\left\|\nabla u^l\right\|^2_2.
\end{align*}
Picking $\delta$ small enough, \eqref{ineq n3} shows for $c^*=\max\{c_2+c_1+1,c_0+1\}$
\begin{equation}\label{ineq:l}
\partial_t \int_{\mathbb{R}^d} u^{n}dx+\int_{\mathbb{R}^d}|\nabla u^l|^2dx\leq 
C_\delta n^{c^*}\left\|u^l\right\|_1^{2+\gamma}.
\end{equation}

As done in \eqref{leftbound}, for some $\gamma'\in (0,2)$ 
\[
\left\|u^n\right\|_1\lesssim  \left\|\nabla u^l\right\|_2^2+\left\| u^l \right\|_1^{\gamma'}\lesssim \left\|\nabla u^l\right\|_2^2+\left\| u^l \right\|_1^{2+\gamma}+1.
\]
To conclude, we find out that
\[\frac{d}{dt}\left\|u^n\right\|_1+\left\|u_1^n\right\|_1\lesssim  1+n^{c^*}\left\|u^l\right\|_1^{2+\gamma}\]
where $c^*>0$ only depends on $s,d,m$.

Finally as in Theorem \ref{propm1},
let $n=n_{k+1}$ where $\{n_{k+1}\}$ is defined in \eqref{sequence} for $k\geq0$, and then $l$ becomes $n_k$. By considering $A_k=\int_{\mathbb{R}^d} u^{n_k}dx$, we conclude the proof after applying Lemma \ref{iteration}.

\end{proof}

Now we proceed to the case when $s\in (1,\frac{d}{2})$.

\begin{theorem}\label{thm:large s}
Theorem \ref{propm1} holds in the regime:  $s\in (1,\frac{d}{2})$, $m>2-2s/d$.
%Suppose $d\geq 3$, $s\in (1,\frac{d}{2})$, $m\geq 1$ and $u_{0}\in L^1(\mathbb{R}^d)\cap L^\infty(\mathbb{R}^d)$ is non-negative. Let $u:=u_{\epsilon,R}$ be the solution to \eqref{eqn: u eps R}. Then there exists a constant $C$ such that for all $R $ large enough and $\epsilon>0$\[\left\|u\right\|_{L^n}(t)\leq C \quad \text{  for all $n\geq 1$ and $t\in \mathbb{R}^+$}.\]
\end{theorem}
\begin{proof}
The proof is separated into two parts: $m< 2$ and $m\geq 2$. 

\begin{flushleft}
$\circ$ Part one, $m<2$.
\end{flushleft}

For $n\geq 3-m$, denote $l=\frac{n+m-1}{2}\geq 1$. We multiply $u^{n-1}$ on both sides of \eqref{eqn:approx} and obtain
\begin{align}
    \partial_t\int_{\mathbb{R}^d} u^n dx
&\leq-mn\int_{\mathbb{R}^d} u^{m-1}\nabla u \nabla u^{n-1}dx+n{\int_{\mathbb{R}^d}\l  V_{s,\epsilon}*u\r u\nabla u^{n-1}dx}\nonumber\\
&\leq-C_m\int_{\mathbb{R}^d} | \nabla u^{l}|^2dx-(n-1)\underbrace{\int_{\mathbb{R}^d}\l  \nabla \cdot V_{s,\epsilon}*u\r  u^{n}dx}_{X:=}.\label{est VR 1}
\end{align}
Let $\chi(x)=\chi_{|x|\leq 1}(x)$ be the indicator function. Let $A_1:=\chi\nabla \cdot V_{s,\epsilon}$ and $A_2:=(1-\chi)\nabla \cdot V_{s,\epsilon}$. It is not hard to see $A_2$ is bounded and
\begin{itemize}
    \item[] 1. $A_1$ is compactly supported and $|A_1|(z)\leq |z|^{-d-2+2s}$. 
    
    \item[] 2. $|A_1|$ bounded in $L^{\frac{d}{d+2-2s'}}(\mathbb{R}^d)$ for all $1<s'<s$.
\end{itemize}
We fix one $s'$ such that 
\[s'\in (1,s)\quad  \text{ and }\quad m>2-\frac{2s'}{d}.\] 
We have
\begin{align*}X&\leq \int_{\mathbb{R}^d} |A_1|*u\, u^{n}dx+ C\int_{\mathbb{R}^{2d}} |A_2|_\infty {u}(y)u^n(x)dxdy=:X_1+X_2.
\end{align*}
By Young's convolution inequality
\begin{equation}
    \label{X1 pq}
    X_1\leq  \|u^n\|_p\|u\|_q \quad \text{ with }p,q\geq 1,\,\frac{1}{p}+\frac{1}{q}=1+\frac{2s'-2}{d}.
\end{equation}%2-\frac{d+2-2s'}{d}=

First consider the case when $\frac{d}{2s'-2}\geq l$. 
We can write
\[\|u^n\|_p=\|u^l\|_{\frac{np}{l}}^\frac{n}{l},\quad \|u\|_q=\|u^l\|_{\frac{q}{l}}^\frac{1}{l}.\]
By Gagliardo-Nirenberg and Young's inequalities,
\begin{equation}\label{X1 p q l}
%\|u^l\|_{r}\leq C\|\nabla u^l\|_2^\alpha\|u^l\|_1^{1-\alpha}
\|u^l\|_{\frac{np}{l}}\leq C\|\nabla u^l\|_2^\alpha\|u^l\|_1^{1-\alpha},\quad   \|u^l\|_{\frac{q}{l}}\leq C\|\nabla u^l\|_2^\beta\|u^l\|_1^{1-\beta}
\end{equation}
where $\alpha,\beta$ are given by
\begin{equation}
    %\frac{1}{r}+(\frac{1}{2}+\frac{1}{d})\alpha=1
\frac{l}{np}+(\frac{1}{2}+\frac{1}{d})\alpha=1,\quad \frac{l}{q}+(\frac{1}{2}+\frac{1}{d})\beta=1.\label{nlp alpha beta}
\end{equation}

If $\frac{d}{2s'-2}\geq n$, take $\frac{np}{l}=\frac{q}{l}=:r$. Then $\alpha=\beta$. According to \eqref{X1 pq}, $r$ can be computed by
\[\frac{1}{r}(\frac{n+1}{l})=1+\frac{2s'-2}{d}.\]
Then 
\[p=\frac{n+1}{n}/(1+\frac{2s'-2}{d})\]
and so $p\geq 1,q\geq 1$ is satisfied due to the assumption $\frac{d}{2s'-2}\geq n$.
Then by \eqref{X1 pq}, \eqref{X1 p q l}
\begin{equation}
    \label{X1 l n}
    X_1\leq C\|\nabla u^l\|_2^{\alpha\frac{n+1}{l}}\|u^l\|_1^{(1-\alpha)\frac{n+1}{l}}.
\end{equation}

We claim
\[\alpha\in [0,1],\quad \alpha\frac{n+1}{l}< 2.\]
It is only nontrivial to verify the second formula of the claim. Let us compute
\begin{align*}
    2-\alpha\frac{n+1}{l}&=2-\frac{n+1}{l}(1-\frac{l}{np})/(\frac{1}{2}+\frac{1}{d})\\
    &=(\frac{1}{2}+\frac{1}{d})^{-1}( 1+\frac{2}{d}-\frac{n+1}{l}+\frac{n+1}{np})\\
    &=(\frac{1}{2}+\frac{1}{d})^{-1}( 2-\frac{n+1}{l}+\frac{2s'}{d})\\
    &\geq (\frac{1}{2}+\frac{1}{d})^{-1}( m-2-\frac{2s'}{d})
\end{align*}
with equality holds when $n=3-m$. We get that$\alpha\frac{n+1}{l}< 2$ is exactly equivalent to
$m> 2-\frac{2s'}{d}.$
Also since the sum of the exponents in \eqref{X1 l n} equals $\frac{n+1}{l}\sim 2+\frac{c}{n}$ for some universal $c>0$, 
we obtain
\[X_1\leq C\|\nabla u^l\|_2^{2-\epsilon}\|u^l\|_1^{\epsilon+\frac{c}{n}}\quad \text{ for some universal }\epsilon>0 \text{ independent of }n.\]
As done several times before, by Young's inequality we derive that
\begin{equation}\label{X1 part 0}X_1\leq Cn^c\|u^l\|_1^{2+\frac{c}{n}}+\frac{\delta}{n}\|\nabla u^l\|^2_2\quad \text{ for some universal constants }C,c.\end{equation}

If $\frac{d}{2s'-2}\in (l,n)$, take $q=\frac{d}{2s'-2}$, $p=1$ and $\alpha,\beta$ satisfying \eqref{nlp alpha beta}. Since \[1>\frac{l}{np},\frac{l}{q}>\frac{1}{2},\]
it is immediately to check that $\alpha,\beta\in (0,1)$. %Obviously $\alpha>0$,\[\alpha<1\iff \frac{1}{2}+\frac{1}{d}> \frac{n-l}{n}\iff \frac{1}{2}-\frac{1}{d}< \frac{l}{n}\]which is true by definition.\[\beta<1\iff \frac{1}{2}+\frac{1}{d}> \frac{q-l}{q}=1-l\frac{2s'-2}{d}\iff \frac{1}{2}-\frac{1}{d}< l\frac{2s'-2}{d}.\]Since $n\geq \frac{d}{2s'-2}$, this is again true.
Then by \eqref{X1 pq}, \eqref{X1 p q l}
\[X_1\leq C\|\nabla u^l\|_2^{\alpha\frac{n}{l}+\beta\frac{1}{l}}\|u^l\|_1^{(1-\alpha)\frac{n}{l}+(1-\beta)\frac{1}{l}}.\]
We claim that
\begin{equation}\label{less 2}\alpha\frac{n}{l}+\beta\frac{1}{l}<2\quad \text{ is equivalent to
}\quad m>2-\frac{2s'}{d}.\end{equation}
Actually $2-\alpha\frac{n}{l}+\beta\frac{1}{l}$ is away from $0$ independent of $n$.
We omit the proof which is a direct computation. Also since \[\alpha\frac{n}{l}+\beta\frac{1}{l}+(1-\alpha)\frac{n}{l}+(1-\beta)\frac{1}{l}=\frac{n+1}{l}\sim 2+\frac{c}{n}\]
for some universal $c>0$,
by H\"{o}lder's inequality in this case, again we have
\begin{equation}
    \label{eqn X1 part1}X_1\leq Cn^c\|u^l\|_1^{2+\frac{c}{n}}+\frac{\delta}{n}\|\nabla u^l\|^2_2.
\end{equation}

Secondly suppose $\frac{d}{2s'-2}\leq l$. We take $p=1$, $q=\frac{d}{2s'-2}\leq l$ in \eqref{X1 pq}. Then since $\|u\|_1$ is bounded, the set $\{u>1\}$ is of finite measure. Thus by Jensen's inequality
\[\int_{\mathbb{R}^d}u^q dx\leq  \int_{u<1}u dx+\int_{u>1}u^q dx\leq C+C\|u\|_l^q.\]
Thus
\begin{equation}
    X_1=  \|u^n\|_1\|u\|_q\leq C\|u^n\|_1(1+\|u^l\|_1^\frac{1}{l}). \label{X_1 un}
\end{equation} 
By Gagliardo-Nirenberg,
\begin{equation*}
\|u^n\|^\frac{l}{n}_1=\|u^l\|_{\frac{n}{l}}\leq C\|\nabla u^l\|_2^\alpha\|u^l\|_1^{1-\alpha}
\end{equation*}
where $\alpha$ is given by
$\frac{l}{n}+(\frac{1}{2}+\frac{1}{d})\alpha=1.
$ From this we get
\begin{align*}
    X_1&\leq C\|\nabla u^l\|_2^{\alpha\frac{n}{l}}(\|u^l\|_1^{(1-\alpha)\frac{n}{l}}+\|u^l\|_1^{(1-\alpha)(\frac{n}{l}+\frac{1}{l})})\\
    &\leq C\|\nabla u^l\|_2^{\alpha\frac{n}{l}}(1+\|u^l\|_1^{(1-\alpha)(\frac{n}{l}+\frac{1}{l})})\\
   & \leq Cn^c+\frac{\delta}{n}\|\nabla u^l\|^2_2+C\|\nabla u^l\|_2^{\alpha\frac{n}{l}}\|u^l\|_1^{(1-\alpha)(\frac{n}{l}+\frac{1}{l})}.
\end{align*}
%By direct computation\begin{align*}\end{align*}
In the last inequality we used that
\[\alpha\frac{n}{l}=\frac{n-l}{l}/(\frac{1}{2}+\frac{1}{d})<\frac{d}{d+2}.\]
And we need
$2-\frac{n}{l}\alpha>0$ to be bounded away from $0$ uniformly in $n$. Actually 
\[2-\frac{n}{l}\alpha=(\frac{1}{2}+\frac{1}{d})^{-1}(2+\frac{2}{d}-\frac{n}{l})\geq (\frac{1}{2}+\frac{1}{d})^{-1}\frac{2}{d}.\]
As before by H\"{o}lder's inequality
\begin{equation}
    \label{eqn X1 part2}X_1\leq Cn^c+Cn^c\|u^l\|_1^{2+\frac{c}{n}}+\frac{2\delta}{n}\|\nabla u^l\|^2_2.
\end{equation}

\medskip

As for $X_2$, note
$X_2\leq C\|u^n\|_1$, therefore it can be handled similarly as we bound \eqref{X_1 un}.

\medskip

In all by \eqref{est VR 1}\eqref{X1 part 0}\eqref{eqn X1 part1}\eqref{eqn X1 part2} and taking $\delta$ to be small, we proved for all $n\geq 3-m$
\begin{equation*}
\partial_t\left\|u^n\right\|_1+c\left\|\nabla u^l\right\|^2_2\leq C_\delta n^c+C_{\delta} n^{c}\left\|u^l\right\|^{2+\frac{c}{n}}_1.\end{equation*}
for some universal constants $C,c>0$. As in \eqref{leftbound}, we can bound $\left\|\nabla u^l\right\|^2_2$ from below. Then as before, taking $n_k=2^{k}(2-m)-1+m$ and $A_k=\|u^{n_k}\|_1$, we end the proof by applying Lemma \ref{iteration}.

\medskip

\begin{flushleft}
$\circ$ Part two, $m\geq 2$.
\end{flushleft}
For $n>1$, we multiply $u^{n-1}_1$ on both sides of \eqref{eqn:approx} where $u_1=(u-1)_+$.
We have
\begin{align*}
    \partial_t\int_{\mathbb{R}^d} u_1^n dx=&
n\int_{\mathbb{R}^d} u_1^{n-1}u_tdx\\
&\leq-mn\int_{\mathbb{R}^d} u^{m-1}\nabla u \nabla u_1^{n-1}dx+n\underbrace{\int_{\mathbb{R}^d}\l  V_{s,\epsilon}*u\r u\nabla u_1^{n-1}dx}_{Y:=}.
\end{align*}
Since $m\geq 2$, we have
\begin{equation}\label{est 5.4}
\partial_t\int_{\mathbb{R}^d} u_1^n dx\leq
- C_m \int_{\mathbb{R}^d}\left|\nabla u_1^{\frac{n}{2}}\right|^2+\left|\nabla u_1^{\frac{n+1}{2}}\right|^2dx+nY.
\end{equation}
For $Y$ using the notation $u=u_1+\tilde{u}$, we have
\begin{align}Y&= \frac{n-1}{n}\int_{\mathbb{R}^d}  V_{s,\epsilon}*u \nabla u_1^{n}dx+ \int_{\mathbb{R}^d} V_{s,\epsilon}*{u}\;\nabla u_1^{n-1}dx\nonumber\\
&\lesssim \int_{\mathbb{R}^d} (-\nabla\cdot V_{s,\epsilon})*u\, u_1^{n}dx+\int_{\mathbb{R}^d} (-\nabla\cdot V_{s,\epsilon})*u\, u_1^{n-1}dx\nonumber\\
&\lesssim \int_{\mathbb{R}^d} |\nabla\cdot V_{s,\epsilon}|*u\, u_1^{n}dx+\int_{\mathbb{R}^d} |\nabla\cdot V_{s,\epsilon}|*u\, (u_1^{n}+u_1)dx\nonumber\\
&\lesssim \int_{\mathbb{R}^d} |\nabla\cdot V_{s,\epsilon}|*(u_1+\tilde{u})\, u_1^{n}dx+1 \quad\quad  (\text{ since } \tilde{u}_1\leq 1,\,  \tilde{u}_1(\cdot,t)\in L^1(\mathbb{R}^d))\nonumber\\
&\lesssim \int_{\mathbb{R}^d} |\nabla\cdot V_{s,\epsilon}|*u_1\, u_1^{n}dx+\int_{\mathbb{R}^d} u_1^{n}dx+1. \nonumber
\end{align}
By Young's convolution inequality, the above
\begin{equation}
    \lesssim   \|u_1^n\|_p\|u_1\|_q+\|u_1^n\|_1+1\label{YYX'}
\end{equation}
where
\[p,q\geq 1,\quad\frac{1}{p}+\frac{1}{q}=1+\frac{2s'-2}{d}.\]

For any $\tilde{m}\in (2-\frac{2s}{d},2)$, let $l=\frac{n+\tilde{m}-1}{2}$.
The goal is to show
\begin{equation*}Y\leq C n^c+C n^{c}\left\|u_1^l\right\|^{2+\frac{c}{n}}_1+\frac{\delta}{n} \left\|\nabla u_1^l\right\|^2_2,\end{equation*}
following from which
\[\partial_t\|u_1^{n}\|^2_1 dx+c\left\|\nabla u_1^{\frac{n}{2}}\right\|^2_2+c\left\|\nabla u_1^{\frac{n+1}{2}}\right\|^2_2\leq
Cn^c+Cn^{c}\|u_1^l\|^{2+\frac{c}{n}}_1.\]
Notice
\[\left\|\nabla u_1^l\right\|^2_2\lesssim \left\|\nabla u_1^{\frac{n}{2}}\right\|^2_2+\left\|\nabla u_1^{\frac{n+1}{2}}\right\|^2_2\]
and the fact that
$\|u_1^n\|_1$ can be handled as done in \eqref{X_1 un}, therefore
the rest of the proof follows from the proof of Part one with $m$ replaced by $\tilde{m}$.

\end{proof}

\section{Existence of Solutions}

In this section, we show existence  of weak solutions to \eqref{main}. We are going to take $\epsilon\to 0$ in equation \eqref{eqn:approx}. Let us consider the case when $s\in(\frac{1}{2},\frac{d}{2}),m>2-\frac{2s}{d}$. 
Let $u_{\epsilon}$ solves \eqref{eqn:approx}. By Theorems \ref{propm1}, \ref{propm2}, \ref{thm:large s}, \[\sup_{\epsilon>0,t\geq 0}\|u_\epsilon(\cdot,t)\|_1+\|u_\epsilon(\cdot,t)\|_\infty<\infty.\]
Then $| V_{s,\epsilon}|$  is locally integrable near the origin and so 
\begin{equation}\label{bound: Vsmu}
\left| V_{s,\epsilon}*u_\epsilon(x)\right|\leq C\int_{|x-y|\leq 1}| V_{s,\epsilon}(x-y)|dy+C\int_{|x-y|>1}u_\epsilon(y)dy<\infty
\end{equation}
independent of $\epsilon$. 
The situation is in some sense better.

\medskip

We have the following theorem.

\begin{thm}\label{thmexistence}
Suppose $\{s\in (\frac{1}{2},\frac{d}{2})\,,m>2-\frac{2s}{d}\}$ and $u_0\in L^1(\mathbb{R}^d)\cap L^\infty(\mathbb{R}^d)$ is non-negative. Then there exists a weak solution $u$ to \eqref{main} with initial data $u_0$ and $u$ preserves the mass.
\end{thm}
Using the estimate given in section \ref{subsection4.1} and the fact that $\| V_{s,\epsilon}*u_{\epsilon}\|_\infty$, $\|\nabla \mathcal{K}_s*u_{\epsilon}\|_\infty$ are uniformly bounded independent of $\epsilon,t$, the proof is almost the same as the proofs in Theorem 1,2,7 in \cite{bedrossian2011}. We omit the proof. The proof of conservation of mass is similar to those in the proof of Theorem \ref{thm:existence msmaller} given below.

\medskip

Let us focus on the situation when $s\leq \frac{1}{2}$. We need the following a prior estimates.

\begin{lemma}\label{lem:priorgrad}
Suppose $s\in (0,\frac{1}{2}]$, $2>m>2-2s/d$ and $u_0\in L^1(\mathbb{R}^d)\cap L^\infty(\mathbb{R}^d)$ is non-negative. Write $u_\epsilon$ as the solution to \eqref{eqn:approx}. Then for any $T>0$, there exists a constant $C_T$ independent of $\epsilon$ such that
\begin{equation}\label{ineq:0T}
    \left\|\nabla u_\epsilon\right\|_{L^2(\mathbb{R}^d\times[0,T])}+ \left\| V_{s,\epsilon}*u_\epsilon\right\|_{L^2(\mathbb{R}^d\times[0,T])}\leq C_T.
\end{equation}
%where $q$ satisfies $1/q=1/2-2s/d$.

\end{lemma}

\begin{proof}
By \eqref{eqn:approx}, after taking $n=3-m$ in \eqref{ineq:l} and doing integration, we find
\[
\int_{\mathbb{R}^d} u^{3-m}_{\epsilon}dx(T)+\iint_{\mathbb{R}^d\times [0,T]}|\nabla u_{\epsilon}|^2dxdt\leq \int_{\mathbb{R}^d} u_{\epsilon}^{3-m}dx(0)
+C |3-m|^{c}\int_0^T\left\|u_{\epsilon}\right\|_1^{2+\gamma}dt.\]
Since $\|u_\epsilon\|_{3-m}(t)$ is uniformly bounded for all $t\geq 0$,
\begin{equation}\label{ineq: nablau_s}
    \left\|\nabla u_\epsilon\right\|^2_{L^2(\mathbb{R}^d\times[0,T])}\leq C+CT
\end{equation} 
where $C$ only depends on $s,m,d$ and $\|u_{\epsilon}\|_1+\|u_{\epsilon}\|_\infty$.

By \eqref{def V} if writing $\rho=|x|$, we have
\[ V_{s,\epsilon}(x):=c\;\zeta_\epsilon(\rho)\nabla_x\; \rho^{-d+2s} \]
with $c$ only depending on $d,s$.
Then we can find $g(\rho):(0,\infty)\rightarrow \mathbb{R}$ such that
\[g'(\rho)=c(-d+2s)\zeta_\epsilon(\rho)\rho^{-d-1+2s}\;\text{ and }\;g(1)=c. \]
Then we have
\[|g(\rho)|\leq C\rho^{-d+2s},\quad |g'(\rho)|\leq C\rho^{-d-1+2s}\]
for some $C>0$ independent of $\epsilon$. Actually for $\rho\in [2\epsilon,1/\epsilon]$, $g(\rho)=c\rho^{-d+2s}$.

Let $\varphi:[0,\infty)\to [0,1]$ be a smooth bump function that $\varphi(\rho)=1$ for $\rho\leq 1$ and $\varphi(\rho)=0$ for $\rho\geq 2$.
Decompose $g$ as $g=g_s+g_b:=\varphi g+(1-\varphi)g$. Then for some constant $C=C(d,s)>0$ independent of $\epsilon$
\begin{equation}
    \label{est gs}
    \|g_s\|_1\leq c\int_{\epsilon}^2 \rho^{-d+2s}\rho^{d-1}d\rho\leq C,
\end{equation} 
\begin{equation}\label{est gb}
    \begin{aligned}
    \|\nabla g_b\|^2_2
    &\leq C\int_{|x|\geq 1} |(\nabla\varphi)\, g|^2(x)+|\nabla g|^2(x) dx\\
    &\leq C  \int_1^2 (\rho^{-d+2s})^2\rho^{d-1}d\rho+C\int_1^\infty(\rho^{-d-1+2s})^2 \rho^{d-1}d\rho\\
    %&\leq C\int_1^\infty \rho^{-d-1+4s} d\rho+C\int_1^\infty \rho^{-d-3+4s} d\rho\\
    &\leq C \quad (\text{ since }d\geq 3, s\leq 1  ).
    \end{aligned}
\end{equation}

It is not hard to see
\begin{align*}
    \left\| V_{s,\epsilon}*u_\epsilon\right\|^2_{L^2(\mathbb{R}^d\times[0,T])}&\leq2\left\|g_s(|x|)*\nabla u_\epsilon\right\|^2_{L^2(\mathbb{R}^d\times[0,T])}+2\left\|\nabla g_b(|x|)* u_\epsilon\right\|^2_{L^2(\mathbb{R}^d\times[0,T])}.\\
&=:2X_1+2X_2.
\end{align*}
Using Lemma \ref{lem:Young}, \eqref{ineq: nablau_s} and \eqref{est gs} gives
\begin{align*}
    X_1&=\int_0^T\left\|g_s*\nabla u_\epsilon\right\|^2_2dt\leq \int_0^T\left\|g_s\right\|^2_{1}\left\|\nabla u_\epsilon\right\|^2_2 \;dt\\
    &= \left\|g_s\right\|^2_{1}\left\|\nabla u_\epsilon\right\|^2_{L^2(\mathbb{R}^d\times[0,T])} <\infty.
\end{align*}
For $X_2$, by \eqref{est gb}
\[X_2\leq  \iint_{\mathbb{R}^d\times[0,T])} \int_{\mathbb{R^d}}|\nabla g_b|^2(x-y)u^2_\epsilon (y)dy dxdt\leq C\iint_{\mathbb{R^d}\times[0,T]} u^2_\epsilon (y)dy dt<\infty.\]
In all, we have
\[\left\| V_{s,\epsilon}*u_\epsilon\right\|_{L^2(\mathbb{R}^d\times[0,T])}\leq C\]
where $C$ only depends on $d,s,T$ and $\|u_\epsilon\|_1,\|u_\epsilon\|_\infty$.

\end{proof}

\begin{thm}\label{thm:existence msmaller}
Suppose the conditions in Lemma \ref{lem:priorgrad} are satisfied. Then there exists a weak solution to \eqref{main} with initial data $u_0$. 
\end{thm}

\begin{proof}
For any small $\epsilon>0$, let $u_\epsilon$ be a solution to \eqref{eqn:approx}. By Theorem \ref{thm:s1/2} there exists a constant $C>0$ independent of $\epsilon$ such that for all $t\geq 0$, \begin{equation}
    \label{thm 4.3 L1}
    \|u_{\epsilon}\|_1(t)+\|u_{\epsilon}\|_\infty(t)\leq C.
\end{equation}

We claim the following uniform tightness of $\{u_\epsilon(\cdot,t)\}_\epsilon$ in $L^1(\mathbb{R}^d)$: for any $T>0$
\begin{equation}
    \label{thm 4.3 tight}\lim_{R\to \infty}\int_{B_{2R}^c} u_{\epsilon}(x,t)dx\to 0
\end{equation}
uniformly in $\epsilon\in (0,1)$ and $t\in [0,T]$.
Take $\varphi=\varphi_{N,R}\in C^\infty_0(\mathbb{R}^d)$ a non-negative cutoff function such that for some $N>>1$ 
\begin{align*}
    &\quad\varphi=0 \text{ when }|x|\leq R, \\ &\quad\varphi=1 \text{ when }NR\geq |x|\geq 2R \,\text{ and }\\
&\varphi\leq 1,\, |\nabla\varphi|\lesssim R^{-1}, |\Delta\varphi|\lesssim R^{-1}.
\end{align*}

 By the equation \eqref{eqn:approx}, for any $t\in (0,T]$
\[\int_{\mathbb{R}^d} u_\epsilon\varphi\, dx (t)=\int_{\mathbb{R}^d} u_\epsilon\varphi\, dx (0)+\underbrace{\int_0^t\int_{\mathbb{R}^d} (\epsilon u_\epsilon+u_\epsilon^m)\Delta \varphi \,dxdt}_{Y_1:=}+\underbrace{\int_0^t\int_{\mathbb{R}^d} (u_\epsilon  V_{s,\epsilon}* u_\epsilon)\nabla\varphi\, dxdt}_{Y_2:=}.\]
Using \eqref{thm 4.3 L1} and $|\Delta\varphi|\lesssim R^{-1}$, $Y_1$ converges to $0$ as $R\to\infty$ uniformly in $\epsilon$ and $t\leq T$. 

Next by H\"{o}lder's inequality and Lemma \ref{lem:priorgrad},
\[ Y_2\leq \left\| V_{s,\epsilon}*u_\epsilon\right\|_2\left\|u_\epsilon\nabla\varphi\right\|_2\leq C_T R^{-1}. \]
Thus independent of $N,\epsilon$ and for all $t\in [0,T]$
\begin{equation*}
   \int_{\mathbb{R}^d}u_{\epsilon}(x,t)\varphi(x) dx\leq \omega(R,T)
\end{equation*}
where $\omega:(\mathbb{R}^{2})^+\to \mathbb{R}^+$ satisfying that 
\[\lim_{R\to\infty}\omega(R,T)=0.\]
If letting $N\to \infty$, we proved \eqref{thm 4.3 tight}.

\medskip

Next by Lemma \ref{lem:priorgrad},
$ \;\left\|\nabla u_\epsilon\right\|_{L^2(\mathbb{R}^d\times[0,T])}\leq C_T.$ This, as well as \eqref{thm 4.3 L1}\eqref{thm 4.3 tight}, implies that $\{u_\epsilon\}_\epsilon$ is precompact in $L^1(0,T,L^1(\mathbb{R}^d))$. The proof follows from the work of \cite{bertozzi2009,bedrossian2011}. Thus by passing $ \epsilon\to 0$ along subsequences, we have $u_\epsilon\to u$ in $L^1(0,T,L^1(\mathbb{R}^d))$.
And we can have
\[u\in L^2(0,T, \dot{H}^1(\mathbb{R}^d)).\]
Then $\nabla(-\Delta)^{-s}u$ is well-defined which is a bounded function in $L^2(\mathbb{R}^d\times [0,T])$. We want to show the weak convergence of $V_{s,\epsilon}* u_\epsilon$ to $\nabla(-\Delta)^{-s}u$.

Let $\xi\in C_0^\infty(\mathbb{R}^d\times [0,T],\mathbb{R}^d)$ be a test function. We have
\[ \iint_{\mathbb{R}^d\times [0,T]} \left( V_{s,\epsilon}*u_\epsilon-\nabla(-\Delta)^{-s}u\right)\xi dxdt=\iint_{\mathbb{R}^d\times [0,T]}  V_{s,\epsilon}*\xi\; u_\epsilon-\nabla\cdot(-\Delta)^{-s}\xi \;u dxdt
\]
\begin{equation}\label{firstterm}\leq C\iint_{\mathbb{R}^d\times [0,T]} \left| V_{s,\epsilon}*\xi\; -\nabla\cdot(-\Delta)^{-s}\xi \right| dxdt+C\iint_{\mathbb{R}^d\times [0,T]} \left|\nabla\cdot(-\Delta)^{-s}\xi \right||u_\epsilon-u| dxdt.
\end{equation}
We used the fact that $u_\epsilon,u$ are uniformly bounded in height. Keep in mind that $u_\epsilon\to u$ in $L^1(\mathbb{R}^d\times [0,T])$.
Then to show the integral converges to $0$ as $\epsilon\to 0$, we only need to estimate the first term of \eqref{firstterm} which is denoted as $X$. Suppose $\max_{t\in [0,T]}\xi(\cdot,t)=0|_{B^c_{R_\xi}}$ for some ${R_\xi}>0$ and then by \eqref{def V},
\begin{align*}
X&\leq C\iiint_{\mathbb{R}^{2d}\times [0,T]} \left|(\zeta_\epsilon(x-y)-1)\nabla_x K_s(x,y)\right|\left| \xi(y,t)-\xi(x,t)
\right| dydxdt\\
&\leq C Lip(\xi)\iiint_{|x-y|\leq 2\epsilon}|x-y|^{-d-1+2s}|x-y| (\chi_{|x|\leq {R_\xi}}+\chi_{|y|\leq {R_\xi}})dxdydt\\
&\leq 2C\,Lip(\xi)\,T \iint_{|x|\leq R_\xi, |z|\leq 2\epsilon}|z|^{-d+2s} dz dx\\
&\leq C'\,  Lip(\xi)\,T \, s^{-1}\,R_\xi^d\,\epsilon^{2s} 
\end{align*}
which converges to $0$ as $\epsilon\to 0$. Then $ V_{s,\epsilon}*u_\epsilon\to \nabla(-\Delta)^{-s}u $ weakly in distribution. 

Again by \eqref{ineq:0T}, we have 
\[\| V_{s,\epsilon}*u_\epsilon\|_{L^2(\mathbb{R}^d\times[0,T])}\leq C_T,\quad\|\nabla (-\Delta)^{-s}u_\epsilon\|_{L^2(\mathbb{R}^d\times[0,T])}\leq C_T\]
So actually we have
\[ V_{s,\epsilon}*u_\epsilon\to \nabla(-\Delta)^{-s}u \text{ weakly in }L^2(\mathbb{R}^d\times [0,T])\]
which gives
\[u_\epsilon  V_{s,\epsilon}*u_\epsilon\to u\nabla(-\Delta)^{-s}u \text{ weakly in }L^1(\mathbb{R}^d\times [0,T]).\]
We proved the existence of weak solutions.

\medskip

Notice \eqref{thm 4.3 tight} and the equation deduce the mass preservation of $u$: for all $t>0$, $\int_{\mathbb{R}^d} u dx=\int_{\mathbb{R}^d} u_0 dx$.
Finally let us mention that the property
$u\in C([0,T], L^1(\mathbb{R}^d))$ follows from \cite{bertozzi2009,bedrossian2011}.

\end{proof}

\section{Uniqueness}\label{secuniq}

In this section, we consider the uniqueness of weak solutions to \eqref{main} in the regime $s>1$. 
In general, the problem is open. %We give the following theorem considering the regime $s>1$. 

\begin{thm}\label{uniquenesssingular}
Suppose $s\in(1,\frac{d}{2})$, $m> 2-\frac{2s}{d}$ and let $u_0\in L^\infty(\mathbb{R}^d)\cap L^1(\mathbb{R}^d)$ be nonnegative. Then weak solutions to \eqref{main} with initial data $u_0$ are unique. 
\end{thm}

%We need the following Lemma, the proof of which can be found in Lemma 2.10 \cite{bedrossian2011}.\begin{lemma}\label{lem grad bound}Under the condition of Theorem \ref{uniquenesssingular}, for any $T>0$, there exists a constant $C$ depending only on $T,s,\|u\|_\infty,\|u_0\|_1$ such that\[\|\nabla u^m\|_{L^2(0,T,L^2(\mathbb{R}^d))}\leq C.\]\end{lemma}

\begin{proof}(of Theorem \ref{uniquenesssingular})
We will follow the approach of \cite{bertozzi2009,bedrossian2011} and estimate the difference of weak solutions in $\dot{H}^{-1}$.
Suppose $u_1,u_2$ are two weak solutions to \eqref{main} with the same initial data $u_0$. For each $t>0$, define $\phi(\cdot,t)$ through
\[\Delta \phi(x,t)=u_1(x,t)-u_2(x,t)
\quad\text{ and }\lim_{|x|\to \infty}\phi(x,t)=0.\]
Then by the equation \[\frac{1}{2}\frac{d}{dt}\int_{\mathbb{R}^d}|\nabla\phi|^2dx=
\int_{\mathbb{R}^d}(\nabla u_1^m-\nabla u_2^m)\nabla\phi\, dx-\int_{\mathbb{R}^d}(u_1-u_2)(\nabla\mathcal{K}_su_1)\nabla\phi\, dx\]\[-\int_{\mathbb{R}^d}u_2(\nabla \mathcal{K}_s(u_1- u_2))\nabla\phi\, dx=: X_1+X_2+X_3.\]
Direct computations yields \[X_1=-\int_{\mathbb{R}^d}(u_1^m-u_2^m)(u_1-u_2)\leq 0.\]

Because $d+2-2s\in (2,d)$ and $u_1(\cdot,t)$ is bounded in $L^1\cap L^\infty$
\begin{equation*}
\left| D^2K_s *u_1\right|(x)\leq C\int_{|x-y|\leq 1}| x-y|^{-d-2+2s}dy+C\int_{|x-y|>1}u_1(y)dy<\infty.
\end{equation*}
We get
\[X_2=-\int_{\mathbb{R}^d}\Delta\phi\,\nabla\mathcal{K}_su_1\nabla\phi\, dx=\int_{\mathbb{R}^d}\nabla\phi\, D^2\mathcal{K}_su_1\nabla\phi\, dx\leq C \int_{\mathbb{R}^d} |D^2\mathcal{K}_su_1||\nabla\phi|^2dx\leq C\|\nabla\phi\|_2^2.
\]

Let $A_1(z)=\chi_{|z|\geq 1}D^2K_s(z)$ and $A_2(z)=\chi_{|z|<1}D^2K_s(z)$. By $A_1(z)$ is bounded and $|A_2|(z)\in L^1(\mathbb{R}^d)$. Then by Young's convolution inequality,
\begin{align*}
    X_3&=\int_{\mathbb{R}^d}u_2(D^2{K}_s*\nabla\phi)\nabla\phi\, dx\\
    &=\int_{\mathbb{R}^d}u_2(A_1(z)*\nabla\phi)\nabla\phi\, dx+\int_{\mathbb{R}^d}u_2(A_2(z)*\nabla\phi)\nabla\phi\, dx\\
    &\leq C\left\|A_1*\nabla\phi\right\|_2\left\|\nabla\phi\right\|_2+C\left\|A_2\right\|_1\left\|\nabla\phi\right\|_2^2\leq C\left\|\nabla\phi\right\|_2^2.
\end{align*}
Letting $\eta(t)=\|\nabla\phi\|_2^2$, we find
\[\frac{d}{dt}\eta(t)\leq C\eta(t).\]
And $\eta(0)=0$ due to $u_1(x,0)-u_2(x,0)=0$.
By Gronwall's inequality $\eta(t)=0$ for all $t\geq 0$ which concludes the proof of the theorem.

\end{proof}

\section{H\"{o}lder Regularity}

In this section we look at the case when $s>1/2$ and prove Theorem \ref{thm holder}. Let $u$ be a solution to \eqref{main} and denote 
\[V(x,t)=\nabla \mathcal{K}_su (x,t).\]
Then we can rewrite the equation as
\begin{equation}\label{eqn:driftV}
    u_t=\Delta u^m + \nabla\cdot({V}u).
\end{equation}

By Theorems \ref{propm1}, \ref{propm2} and \ref{thm:large s}, in the subcritical regime, $u$ is uniformly bounded in $ L^\infty(\mathbb{R}^d\times [0,\infty))$ and $\|u(\cdot,t)\|_1=\|u_0\|_1<\infty$. Thus 
\begin{equation}\label{bound of V}
    \begin{aligned}
    |\nabla\mathcal{K}_su|(x,t)&\lesssim \int_{\mathbb{R}^d}
|x-y|^{-d-1+2s}u(y,t)dy\\
&\lesssim \int_{|x-y|\leq 1}
|y|^{-d-1+2s}dy+\int_{|x-y|\geq 1}u(y,t)dy \\
&<C \quad (\text{ only depending on }d,s, \|u_0\|_1, \|u\|_\infty).
\end{aligned}
\end{equation}
Therefore ${V}(x,t)$ is uniformly bounded.
Let us consider \eqref{eqn:driftV} and the notion of solutions is the same as Definition \ref{def1.1} after replacing $\mathcal{K}_su $ by $V$.
We give both the interior regularity and the regularity up to $t=0$ results of solutions to \eqref{eqn:driftV}. Here we only need $m>1$.

\begin{thm}\label{tequal0}
Suppose $m>1$ and $V$ is a bounded measurable vector field in $\mathbb{R}^d\times [0,1]$. Let $u$ be a bounded weak solution to \eqref{eqn:driftV} in $B_1\times [0,1]$. Then the following holds
\begin{itemize}
\item[(a)] $u$ is H\"{o}lder continuous in $B_{\frac{1}{2}} (\frac{1}{2},1]$.

\medskip

\item[(b)] If $u(\cdot,0)$ is H\"{o}lder continuous in space, then $u$ is H\"{o}lder continuous in $B_{\frac{1}{2}}\times [0,1]$.
\end{itemize}
\end{thm}

\begin{proof}

Part (a) follows from  Theorem 4.1 \cite{kimpaul}. We prove part (b) below.

\medskip

Let $\alpha=\frac{m-1}{m}$ and for some $r,w>0$ recall \eqref{Qrc}
\[Q^0(r,w^{-\alpha}):=\left\{(x,t),|x|\leq r, t\in [0,w^{-\alpha}r^2]\right\}.\]
For simplicity we write $Q^0:=Q^0(1,1)$.  Let  $v:=u^m$ which then solves
\[(v^\frac{1}{m})_t=\Delta v + \nabla\cdot({V}v^\frac{1}{m})\]
with initial data $v_0(x):=u^m(x,0)$.
Denote $M$ as the supremum of $v$ in $Q^0$. 

Fix any point $x\in B_\frac{1}{2}$ and without loss of generality, we can assume $x=0$. 
The first goal is to obtain
\begin{equation}\label{holder}
\eta^kM \geq {\rm osc}_{Q^0(a^kr, b^{2k})} v \hbox{ for all integers } k,
\end{equation}
where $ a,b,\eta\in (0,1)$ only depends on $M, m ,d$, $\|V\|_{L^\infty(Q^0)}$, $\gamma$ and $\gamma$-H\"{o}lder norm of $v_0 $, which will be called as universal constants from now on and within this section.

\medskip

We need two lemmas which regard oscillation reduction. The first one implies that under a suitable assumption the solution is bounded away from $0$ with certain amount. The other shows that if the assumption is not satisfied, then the supremum of the solution decreases once we look at a smaller parabolic neighborhood of $0$.

\medskip

Take $w=M$.
We start with some $Q^0(r,w^{-\alpha})$ with $0<r\leq \frac{1}{2}$ such
 that 
\begin{equation}\label{improvement1}
Q^0(r,w^{-\alpha})\subset Q^0_\frac{1}{2}\;\text{ and by definition }w \geq osc_{Q^0(r,w^{-\alpha})}v.
\end{equation}
Denote
\[ M^-=\inf\left\{v, (x,t)\in Q^0(r,w^{-\alpha})\right\},\;M^+=\sup\left\{v, (x,t)\in Q^0(r,w^{-\alpha})\right\}.\]

\begin{flushleft}
\textit{Claim 1:} 
Suppose \eqref{improvement1}, $v(0,0)\geq M^-+{w}/{4}$ and
\begin{equation}\label{improvement}
w \geq osc_{Q^0(r,w^{-\alpha})}v\;\text{ and }\;M^-\leq \frac{w}{4}.
\end{equation}  
Then there exist universal constants $c_1,c_2\in(0,1), l\geq 1$ such that the following holds: if $0<r<c_1w^{c_2}$, then
\[v|_{Q^0(\frac{r}{2},{w}^{-\alpha})}\geq M^- + \frac{w}{2^l}.\]\end{flushleft}

\begin{flushleft}
\textit{Claim 2:}
Suppose $v(0,0)\leq M^-+w/4$ and \eqref{improvement1}\eqref{improvement} hold.
Then there exist universal constants $c_1,c_2, \eta\in (0,1)$ such that the following is true: 
if $0<r< c_1w^{c_2}$, then
\[v|_{Q^0\l\frac{r}{2}, w^{-\alpha}\r}\leq M^-+\eta w.\]\end{flushleft}
%The constant $\eta$ is independent of $w$ which may depend on $c_0$.

The proofs of \textit{Claim 1} and \textit{Claim 2} are similar. The proof of \textit{Claim 1} is parallel to Proposition 4.4 in \cite{kimpaul} where the interior H\"{o}lder continuity property of \eqref{eqn:driftV} is proved while \textit{Claim 2} parallel to Proposition 4.6 \cite{kimpaul}. We also refer readers to Section 3.11 in \cite{dibenedettobook} where continuity of solutions up to time $0$ is proved when $V=0$.

\medskip

Let us only outline the proof of \textit{Claim 2}: 
If $M^+-M^-\leq 3w/4$, then there is nothing to prove since we can simply take $\eta=\frac{3}{4}$. Then we assume $M^+\geq M^-+3w/4$.

Since $v_0$ is H\"{o}lder continuous, we have \[osc_{\mathbb{R}^d}v_0 \leq r^\gamma\text{ for some } \gamma>0.\] 
By selecting $c_1,c_2$ appropriately and requiring $r\leq c_1 w^{c_2}$, and using $v(0,0)\leq M^-+w/4$, we obtain \[\sup |v(\cdot,0)|_{\mathbb{R}^d}\leq M^-+w/2\leq M^+-w/4.\]
This is different from the interior estimates in \cite{kimpaul} which mainly takes the place of Lemma 4.9 \cite{kimpaul}.

Next by proceeding as in Lemma 4.10 \cite{kimpaul}, we can show the following: 
Assume \eqref{improvement} is satisfied and $r<c_1w^{c_2}$ for some universal constants $c_1,c_2>0$. There exists a universal constant $e>0$ such that if for any fixed $l_0\geq 1$ we have
\begin{equation}\label{condition}\left|\left\{(x,t)\in Q^0(r,w^{-\alpha}),v\geq M^+-2^{-l_0}{w}\right\}\right|\leq e\;|Q^0(r,w^{-\alpha})|,\end{equation}
then
\[v|_{Q^0\l\frac{r}{2}, w^{-\alpha}\r}\leq M^+-2^{-l_0-1} w.\]

Since the choice of $e$ is independent of $l_0,w$, we fix it and try to find $l_0$ which only depends on $e$ and universal constants such that the condition \eqref{condition} is satisfied. This is done similarly as in Lemma 4.9 \cite{kimpaul}.

By the assumption \eqref{improvement1}, $M^+\leq M^-+w$. Therefore
\[v|_{Q^0(\frac{r}{2},w^{-\alpha})}\leq M^+-2^{-l_0-1}{w} \leq M^-+(1-2^{-l_0-1})w.\]
We finished the proof of the claim. $\Box $
\medskip

Now we go back to the proof of the Theorem. We refer readers to the proof Theorem 4.1 \cite{kimpaul}.

 Let us start with a given pair of $(r_0,w_0)=(r,w)$. Below we will generate a sequence of pairs $(r_n, w_n)$ that satisfies \eqref{improvement1}. For each $n$ and the given pair $(r_n, w_n)$ let us denote
$$
M^-_n:=\inf_{Q^0(r_n,w_n^{-\alpha})}v, \quad M^+_n:=\sup_{Q^0(r_n,w_n^{-\alpha})} v.
$$
 
Let $c_1$ and $c_2$ be as given in the claims. For each given pair $(r_n, w_n)$ the next pair $(r_{n+1}, w_{n+1})$ is generated depending on the following cases.

\begin{itemize}
\item[Case 1:] if $r_n> c_1w_n^{c_2}$, the situation is in some sense better since the oscillation is under control. In order to apply the preceding scheme, let $w_{n+1}=w_n, r_{n+1}=\frac{1}{2}r_n$, and we repeat until it falls into Case 2 or 3.

\item[Case 2:] if $r_n\leq c_1 w_n^{c_2}$ and either $M^-_n \geq \frac{w_n}{4}$ or  $v(0,0)\geq M_n^-+{w}/{4}$, we claim $v\in [w_n/4, M_n^+]$ in $Q^0(\frac{3r_n}{4},w_n^{-\alpha})$. This is trivial if $M_n^-\geq \frac{w_n}{4}$, otherwise we use \textit{Claim 1}.  Then from classical regularity theory for parabolic equations, it follows that \eqref{holder} holds for $k\geq n$. 

\medskip

\item[Case 3:] We are left with the case $r_n\leq c_1w_n^{c_2}$, $M^-_n < \frac{w_n}{4}$ and $v(0,0)< M_n^-+{w}/{4}$. In this case
 \textit{Claim 2} yields constant $0<\eta<1$ which are independent of $w$ such that
\begin{equation}\label{etaiteration}osc_{Q^0(\frac{r_n}{2}, w_n^{-\alpha})}v\leq \eta w_n.
\end{equation}
We choose \[w_{n+1}:=\eta w_n, \quad r_{n+1}:=c_3 r_n.\]
Here $c_3^2:=\frac{1}{4}\eta^{{\alpha}}$ is chosen such that $Q^0(r_{n+1},w_{n+1}^{-\alpha})\subset Q^0(\frac{r_n}{2},w_n^{-\alpha})$. 
%\[c_3r_n\leq \frac{r_n}{2},\quad (\eta w_n)^{-\alpha}c_3^2r_n^2\leq \frac{c_0}{2}w_n^{-\alpha}\l\frac{r_n}{2}\r^2.\]
From this choice of $c_3$ and \eqref{etaiteration} it follows that  \eqref{improvement1} holds for $(r_{n+1}, w_{n+1})$. 

Suppose Case 3 is iterated for $n$ times. Then inside $\{|x|<c_3^nr, t\in(-w^{-\alpha} 2^{-2n+1}r^2,0)\}$, the oscillation of $v$ is bounded by $\eta^n w$. This yields \eqref{holder} for $k=n$.

\end{itemize}

\medskip

Recall the notation \eqref{distance in d plus one}. By \eqref{holder}, there is a H\"{o}lder modulus of continuity $\rho$ i.e. 
\[\rho:\mathbb{R}^+\to\mathbb{R}^+,\,\rho(r)\leq \tilde{C}r^{\tilde{\alpha}}\]
such that for any $x,y\in\mathbb{R}^d$, $t\in\mathbb{R}^+$
\[|u(x,t)-u(y,0)|\leq \rho(|(x,0),(y,t)|).\]
The H\"{o}lder norm only depends on universal constants.

By part (a), there is a H\"{o}lder modulus of continuity depending only on universal constants, without loss of generality suppose it is $\rho$ again, that if $u$ solves the equation in $B_1\times [0,1]$ then $u$ is $\rho$-modulus continuous in $B_{\frac{1}{2}}(\frac{1}{2},1]$.

Now take any $(x_0,t_0),(y,s)\in B_{\frac{1}{2}}[0,1]$. Suppose $t\geq s$ and $|(x_0,t_0),(x_0,0)|=2r$. 
If $ |(x_0,t_0),(y,s)|\geq r/2$, then
\[|(y,s),(x_0,0)|\leq  \frac{5}{4}|(x_0,t_0),(x_0,0)|=\frac{5}{2}r.\]
Thus
\begin{align*}
    |u(x_0,t_0)-u(y,s)|&\leq |u(x_0,t_0)-u(x,0)|+|u(y,s)-u(x_0,0)|\\
    &\leq \rho(2r)+\rho(\frac{5}{2}r)\leq C r^{\tilde{\alpha}}.
\end{align*}

If $ |(x_0,t_0),(y,s)|\leq r/2$, define
$v(x,t)=r^{-\frac{1}{m-1}}u(x_0+rx,t_0+rt)$. Then if assuming
$v(y',s')=r^{-\frac{1}{m-1}}u(y,s)$, $(y',s')\in B_{\frac{1}{2}}\times [-\frac{1}{2},0]$.
We have $v$ solves
\[\partial_t v=\Delta v^m+\nabla\cdot(\tilde{V}v) \text{ in }B_1\times [-1,0]\]
with $\tilde{V}(x):=V(x_0+rx,t_0+rt)$ which has the same $L^\infty$ bound as $V$ does. So by the interior estimate
\[|v(y',s')-v(0,0)|\leq \rho(|(y',s')|).\]
Therefore
\[|u(x_0,t_0)-u(y,s)|\leq r^{\frac{1}{m-1}}\rho\left(\frac{|(x_0,t_0),(y,s)|}{r}\right). \]
This illustrates that the solution is H\"{o}lder continuous in $ B_{\frac{1}{2}}[0,1]$ with H\"{o}lder exponent $\min\{\frac{1}{m-1},\tilde{\alpha}\}$.

\end{proof}

Recall \eqref{bound of V}, as a corollary of Theorem~\ref{tequal0}, we established Theorem~\ref{thm holder}.

\appendix

\section{Proof of Lemma \ref{gagliardo}}

For $q>1$, the result is covered by Corollary 1.5 \cite{BaoW}. We only need to consider the case when $q=1$ and $s\ne 0$. Fix $s,\alpha,r,p$ that $\alpha>s$ and \eqref{condition1}, \eqref{condition3} are satisfied. 
Notice \eqref{condition1} is equivalent to
\[
    \frac{1}{p}=\frac{s}{d}-(1+\frac{1}{d}-\frac{1}{r})\alpha+1.\]
Thus we can take $\alpha',q'$ such that 
\[s\leq \alpha'<\alpha,\, 1<q'<r,\]
\begin{equation}\label{app:q'p}
    \frac{1}{p}=\frac{s}{d}+(\frac{1}{r}-\frac{1}{d})\alpha'+\frac{1-\alpha'}{q'}.
\end{equation}
By Corollary 1.5 \cite{BaoW}
\begin{equation}\label{app:inq1}
    \left\||\nabla|^s u\right\|_p\leq C\left\|\nabla u\right\|_r^{\alpha'}\left\|u\right\|_{q'}^{1-\alpha'}.
\end{equation}
By the classical Gagliardo-Nirenberg inequality
\begin{equation}\label{app:inq2}\left\| u\right\|_{q'}\leq C\left\|\nabla u\right\|_r^{\beta}\left\|u\right\|_{1}^{1-\beta}\end{equation}
where $\beta$ satisfies
\[\frac{1}{q'}=\left(\frac{1}{r}-\frac{1}{d}\right)\beta+{1-\beta}.\]
Since $q'<r$, \[\beta=(1-\frac{1}{q'})/(1+\frac{1}{d}-\frac{1}{r})\in (0,1).\]
%\left(1-\frac{1}{q'}\right)\left(1+\frac{1}{d}-\frac{1}{r}\right)\right)

By \eqref{app:q'p} and simple calculations
\begin{align*}
    \frac{1}{p}&=\frac{s}{d}+\left(\frac{1}{r}-\frac{1}{d}\right)\alpha'+\left({1-\alpha'}\right)\left(\left(\frac{1}{r}-\frac{1}{d}\right)\beta+{1-\beta}\right)\\
    &=\frac{s}{d}+\left(\frac{1}{r}-\frac{1}{d}\right)\left(\alpha'+\beta-\alpha'\beta\right)+\left({1-\alpha'}\right)(1-\beta).
    \end{align*}
Comparing this with \eqref{condition1}, we obtain
\[\alpha=\alpha'+\beta-\alpha'\beta.\]

Finally plugging \eqref{app:inq2} into \eqref{app:inq1} gives
\[\left\||\nabla|^s u\right\|_p\leq C\left\|\nabla u\right\|_r^{\alpha}\left\|u\right\|_{1}^{1-\alpha}.\]

\medskip

%\bibliography{KS}{}
%\bibliographystyle{plain}

\end{document}